\documentclass[11pt]{article}

\usepackage[T1]{fontenc}
\usepackage{mathpazo} 
\usepackage{microtype} 
\usepackage{amsmath,amstext,amsopn,amssymb, amsfonts, amsthm, mathtools} 
\usepackage{enumitem} 
\usepackage{geometry} 
\usepackage[dvipsnames]{xcolor}
\usepackage{bbm}
\usepackage{mathscinet}
\usepackage{xfrac}
\usepackage{multicol}
\usepackage{graphicx}
\usepackage{mathrsfs}
\usepackage{pifont}
\usepackage{relsize}
\usepackage{verbatim}

\setlist[itemize]{leftmargin=2em, itemsep=0.4em, topsep=0.6em, label=\scalebox{0.6}{\ensuremath{\blacksquare}}}

\newcommand{\unit}{1\!\!1}

\newcommand{\mbz}{\mathbb{Z}}
\newcommand{\mbr}{\mathbb{R}}
\newcommand{\mbn}{\mathbb{N}}
\newcommand{\mbm}{\mathbb{M}}

\newcommand{\mbb}{\mathbb{B}}

\newcommand{\mcd}{\mathcal{D}}
\newcommand{\mca}{\mathcal{A}}

\newcommand{\mcg}{\mathcal{G}}
\newcommand{\mcq}{\mathcal{Q}}

\newcommand{\scb}{\mathscr{B}}
\newcommand{\scm}{\mathscr{M}}
\newcommand{\scs}{\mathscr{S}}

\newcommand{\bfg}{\mathbf{g}}
\newcommand{\bfb}{\mathbf{B}}
\newcommand{\bfm}{\mathbf{M}}
\newcommand{\La}{\langle }
\newcommand{\Ra}{\rangle }
\newcommand{\mfs}{\mathfrak{S}}

\newcommand{\Lint}{\mathtt{L}_{\mathtt{int}}}

\theoremstyle{plain}
\newtheorem{theorem}{Theorem}
\newtheorem{lemma}{Lemma}
\newtheorem{proposition}{Proposition}
\newtheorem{corollary}{Corollary}

\theoremstyle{definition}

\theoremstyle{remark}
\newtheorem{remark}[theorem]{Remark}

\usepackage{titlesec}
\titleformat{\section}[block]{\normalfont\large\bfseries}{\thesection.}{0.5em}{}
\titleformat{\subsection}[runin]{\normalfont\bfseries}{\thesubsection.}{0.5em}{}[.]

\usepackage{fancyhdr}
\title{The Bellman Function for Level Sets of Sparse Operators}
\author{
  Irina Holmes Fay, $\:$
  Zachary H. Pence, $\:$
    John Freeland Small, $\:$
and $\:$
  Xiaokun Zhou
}
\date{} 

\begin{document}

\maketitle

\begin{abstract}
We investigate weak-type $(1, 1)$  boundedness of sparse operators with respect to Lebesgue measure. Specifically, we find the Bellman function maximizing level sets of sparse operators (localized to an interval) and use this to find the exact weak-$(1,1)$ norm of these sparse operators.
\end{abstract}


\vspace{0.3in}

Let $I$ be a real interval -- where by ``interval'' we always mean one of the form $[a,b)$. We work with binary sequences $\alpha$ adapted to $I$, i.e. indexed by the dyadic subintervals of $I$:
$$\alpha = \{\alpha_J\}_{J\in\mcd(I)}; \:\: \alpha_J \in \{0,1\}, \text{ for all } J \in \mcd(I),$$
where the dyadic lattice $\mcd(I)$ is the collection of subintervals of $I$ defined inductively:
    $$
    \mcd(I) := \bigcup_{k=0}^\infty \mcd_k(I), \text{ where } 
    \mcd_0(I):=\{I\} \text{ and } \mcd_{k+1}(I) := \{J_-, J_+ :\: J\in\mcd_k(I)\} \text{ for } k\geq 0,
    $$
where $J_-$ and $J_+$ denote the left and right halves of $J$, respectively. We can think of $\alpha$ as a selection procedure, and denote:
   $$
    \scs_\alpha := \{K\in \mcd(I): \: \alpha_K = 1\},
    $$
the collection of dyadic subintervals of $I$ ``selected'' by $\alpha$.
For every $J\in\mcd(I)$, define:
    $$
    A(\alpha; \: J) := \frac{1}{|J|}\sum_{K\in\mcd(J)}\alpha_K |K|. 
    $$
This is an averaging quantity, measuring ``how much'' of $\alpha$ is packed within $J$, relative to the size of $J$. It is analogous to averaging a function $f$ over $J$, which we denote by:
   $$
    \La f\Ra_J := \frac{1}{|J|} \int_J f(t)\,dt.
    $$

For some $C>1$, we say $\alpha$ is \textbf{$C$-Carleson} if and only if
    $
    \sup_{J\in\mcd(I)} A(\alpha; \: J) \leq C.
    $
Equivalently, the collection $\scs_\alpha$ is  $\sfrac{1}{C}$-sparse (see \cite{LNBook} for details). Let $\mfs_C(I)$ denote the set of all $C$-Carleson binary sequences adapted to $I$. In practical applications, the specific constant $C$ typically does not matter, as one can run equivalent sparse domination arguments with any $C>1$; so, in this paper we work exclusively with $2$-Carleson sequences $\alpha \in \mfs_2(I)$.

\vspace{0.2in}

For  $r>0$ and $\alpha \in \mfs_2(I)$, define the (localized) sparse operator:
$$
\mca_{\alpha, r}f(t):= \left(\sum_{J\in\mcd(I)} \alpha_J \: \langle |f|\rangle_J^r \: \unit_J(t)\right)^{1/r}.
$$
Some notable examples include the classical sparse operator (for $r=1$), used in domination arguments for many linear operators, and the square-sparse operator (for $r=2$), which dominates the dyadic square function.

We will also consider the closely-related power-mean sparse operator 
$$
\mcq_{\alpha, p}f(t) := \sum_{J\in\mcd(I)} \alpha_J \La |f|^p\Ra_J^{\sfrac{1}{p}} \unit_J(t),
$$
for $p\geq 1$ (see Section \ref{sS:power-mean}).

\begin{center}
  {\Large\ding{167}} 
\end{center}

We investigate 
weak-type $(1,1)$ boundedness of $\mca_{\alpha, r}$, with respect to Lebesgue measure. If there is a $C_{\alpha, r}>0$ such that
\begin{equation}
    \label{e:W11}
|\{t\in I: \mca_{\alpha,r}f(t)\geq\lambda\}| \leq \frac{C_{\alpha, r}}{\lambda} \int_I |f(t)|\,dt, \text{ for all } f\in L^1(I) \: \text{ and all } \: \lambda>0,  
\end{equation}
then the optimal constant $C_{\alpha, r}$ above is the operator norm 
    $$ \|\mca_{\alpha, r}\|_{(1,1)} :=  \|\mca_{\alpha, r} : L^1(I) \rightarrow L^{1,\infty}(I)\|.$$
Now, in \eqref{e:W11}, we have $|f|$ on the right-hand side, and the operator $\mca_{\alpha,r}$ also only sees $|f|$ -- so there is no loss of generality if we take $f\geq 0$ in \eqref{e:W11}. We also adjust for working instead with $\mca_{\alpha,r}^r$, and rewrite \eqref{e:W11} as
$$
\|\mca_{\alpha, r}\|_{(1,1)} = \sup_{f, \lambda} \:
\frac{\lambda^{\sfrac{1}{r}}}{\La f\Ra_I} \: \frac{1}{|I|} |\{t\in I: \mca^r_{\alpha,r}f(t)\geq\lambda\}|,
$$
where supremum is over all $f\geq 0$ on $I$ with $0<\La f\Ra_I <\infty$, and all $\lambda>0$. If we take supremum over all sequences $\alpha\in\mfs_2(I)$ on both sides of the previous equality, and index the functions $f$ by $x:=\La f\Ra_I \in (0,\infty)$, and the $\alpha$'s by $A:=A(\alpha; \: I) \in [0,2]$, we obtain:
    \begin{equation}
        \label{e:W11-withB}
        \sup_{\alpha \in \mfs_2(I)} \|\mca_{\alpha, r}\|_{(1,1)} = \sup_{\substack{x>0, \: \lambda>0;\\ A\in [0,2]}} 
        \: \frac{\lambda^{\sfrac{1}{r}}}{x} \mbb_r(x,A,\lambda),
    \end{equation}
where for every $x\geq 0$, $A\in [0,2]$, and $\lambda\in\mbr$, we define the \textbf{Bellman function}:
\begin{equation} \label{E:B-def1}
\mbb_r (x,A,\lambda) := \sup_{f, \alpha} \frac{1}{|I|} |\{t\in I: \mca_{\alpha,r}^r f(t) \geq\lambda\}|,
\end{equation}
where supremum is over all:
\begin{itemize}
\item non-negative functions $f\geq 0$ in $L^1(I)$ with fixed average 
    $    \langle f \rangle_I = x;   $
\item sequences $\alpha\in\mfs_2(I)$ with 
    $    A(\alpha; I) = A.    $
\end{itemize}
\noindent In this paper we find $\mbb_r$, and the quantity in \eqref{e:W11-withB}.
As shown in Section \ref{sS:power-mean}, this immediately gives us weak-type $(p,p)$ norms for the power-mean sparse operators $\mcq_{\alpha,p}$.

\vspace{0.2in}
\begin{minipage}[c]{0.5\linewidth}
    \begin{proposition}
        \label{prop:bound}
    For $r>0$,
    $$
         \sup_{\alpha \in \mfs_2(I)} \|\mca_{\alpha, r}\|_{(1,1)} = C(r) := \left(
    \frac{2^{r+1}-1}{2^r-1}
    \right)^{\sfrac{1}{r}}.
    $$
    \end{proposition}

    \begin{corollary}
        \label{cor:power-mean}
        For $p\geq 1$:
        $$
        \sup_{\alpha \in \mfs_2(I)} \|\mcq_{\alpha, p}\|_{(p,p)} = C(\sfrac{1}{p})^{\sfrac{1}{p}} = 
        \frac{2 \cdot 2^{\sfrac{1}{p}}-1}{2^{\sfrac{1}{p}}-1}.
        $$
    \end{corollary}
\end{minipage}
\begin{minipage}[c]{0.5\linewidth}
\begin{center}
\includegraphics[width=0.85\linewidth]{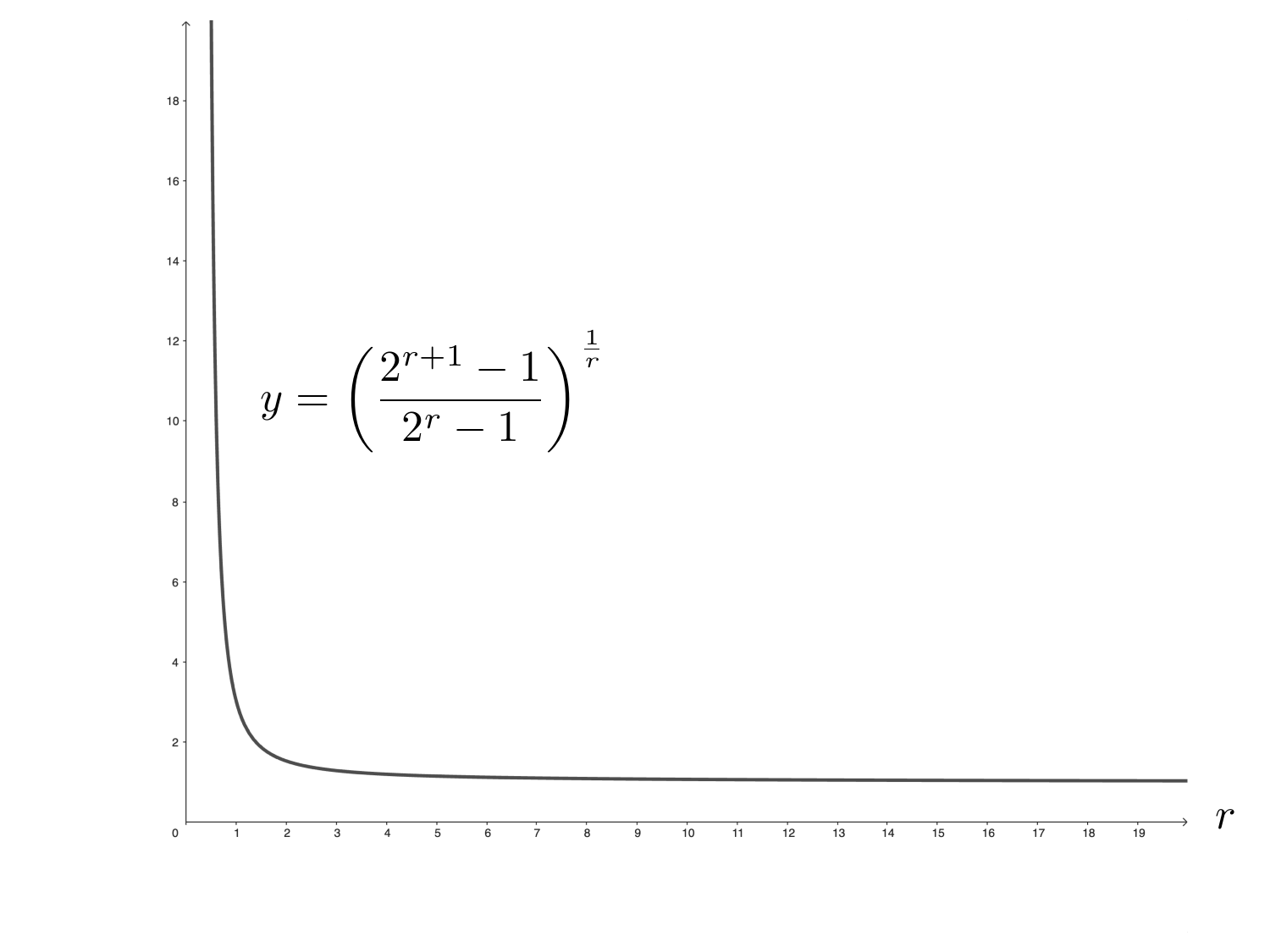}
\end{center}
\end{minipage}

\vspace{0.2in}

Our approach is inspired by \cite{HRS1}, where the Bellman function for the \textit{restricted} weak-type estimate is found. In that case, one works with sparse operators acting on indicator functions, and the Bellman function has three variables. Since we are looking at sparse operators acting on non-negative functions, we have an extra homogeneity which will reduce the problem to two variables.
A strong-type estimate for the case $r=2$ was found in \cite{MR2500520}. In \cite{MR3812865} and \cite{Reznikov2019}, similar weak-type estimates are investigated for martingale transforms.


\section{Main Results}

Let $r>0$. We describe the Bellman function defined in \eqref{E:B-def1}
in terms of its values at $\lambda=1$. We shall see that $\mbb_r$ satisfies a homogeneity relationship which allows us to write
$$
\mbb_r (x,A,\lambda)= \mbm_r \big(x \lambda^{-\sfrac{1}{r}}, \:\: A\big),
$$
where the function 
$$\mbm_r(\omega, A):=\mbb_r(\omega, A, 1)
: \Omega_\mbm\rightarrow[0,1], \text{ defined on } \Omega_\mbm:=[0,\infty) \times [0,2],
$$ 
is graphed below in Figure \ref{fig:Mgraphr}.

\begin{figure}[h!]
\begin{center}
\includegraphics[width=0.7\linewidth]{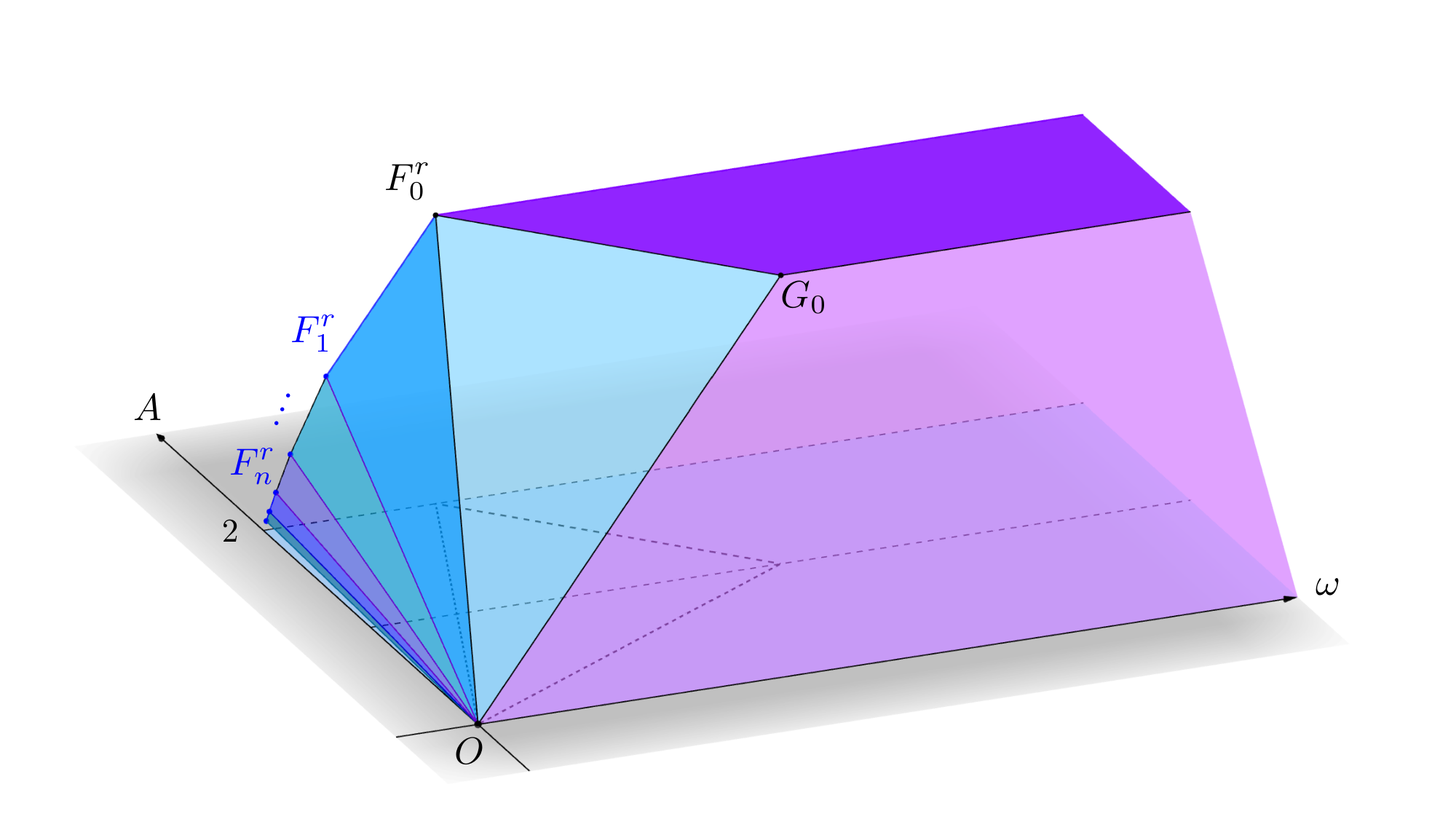}
\end{center}
\caption{The surface $z = \mbm_r(\omega, A)$ (shown here with $r=0.8$).}
\label{fig:Mgraphr}
\end{figure}

The most important critical points on this graph are 
    $
    O:=(0,0,0), \: G_0 := (1,1,1),
    $
and the sequence of points $\{F_n^r\}_{n\geq 0}$ given by
    $$
    F_n^r := \left(\omega_n(r), \: 2, \: \frac{1}{2^n}\right) \: \text{ for all } \: n\geq 0,
    $$
where $\{\omega_n(r)\}_{n\geq 0}$ is the decreasing sequence
defined by
    \begin{equation}
        \label{e:def-omegar-exp}
        \omega_n(r) = \left( \frac{2^r-1}{2^{nr} (2^{r+1}-1)-1} \right)^{\sfrac{1}{r}}, \:\: \text{ for all } n \geq 0.
    \end{equation}
The surface $z=\mbm_r(\omega, A)$ is piecewise linear, and consists of:
    \begin{itemize}
        \item Two infinite trapezoids: one in the horizontal plane $z=1$, bounded by the line segment $[G_0 F_0^r]$ and the rays parallel to the $\omega$-axis originating at $G_0$ and $F_0^r$, and the other in the $z=A$ plane, bounded by $[O G_0]$, the $\omega$-axis, and the parallel ray originating at $G_0$.
        \item The triangular surface bounded by $\Delta [O G_0 F_0^r]$.
        \item The sequence of triangular surfaces bounded by 
            $
            \left\{\Delta [O F^r_{n-1} F_n^r]\right\}_{n\geq 1}.
            $
    \end{itemize}

\noindent To be precise,  
\begin{equation}
\label{e:M-def-r}
	\mbm_r(\omega,A):=\begin{cases}
        1, & \text{ if } (\omega, A) \in \Sigma_0(r)\\[2pt]
        A, & \text{ if } (\omega, A) \in \Sigma_1(r) \\[5pt]
        \frac{\omega+\left(1-\omega_0(r)\right)A}{2-\omega_0(r)}, & \text{ if } (\omega, A) \in \Delta_0(r) \\[5pt]
	\frac{A}{2} \cdot \mathfrak{f}_r\left(\frac{2\omega}{A}\right), & \text{ if } (\omega, A) \in \Delta(r),
	\end{cases}
\end{equation}
where the regions partitioning the domain $\Omega_\mbm$ are described below in Figure \ref{fig:Mregions1}, and, for $0<\omega<\omega_0(r)$, 
the $A=2$ boundary function $\mathfrak{f}_r(\omega)=\mbm_r(\omega, 2)$ is defined by linear interpolation between the points
    $
    \left(\omega_n(r),\: 2^{-n}\right):
    $
    \begin{equation}
        \label{e:f-def-r}
    \mathfrak{f}_r(\omega) := \begin{cases}
        0, & \:\text{ if }\: \omega=0\\
        \Lint[\: (\omega_n(r), 2^{-n}); \: n\in \mbz_{\geq 0}\:], & \:\text{ if }\: 0<\omega<\omega_0(r).
    \end{cases}
    \end{equation}

\begin{figure}[h]
\begin{center}
\includegraphics[width=0.9\linewidth]{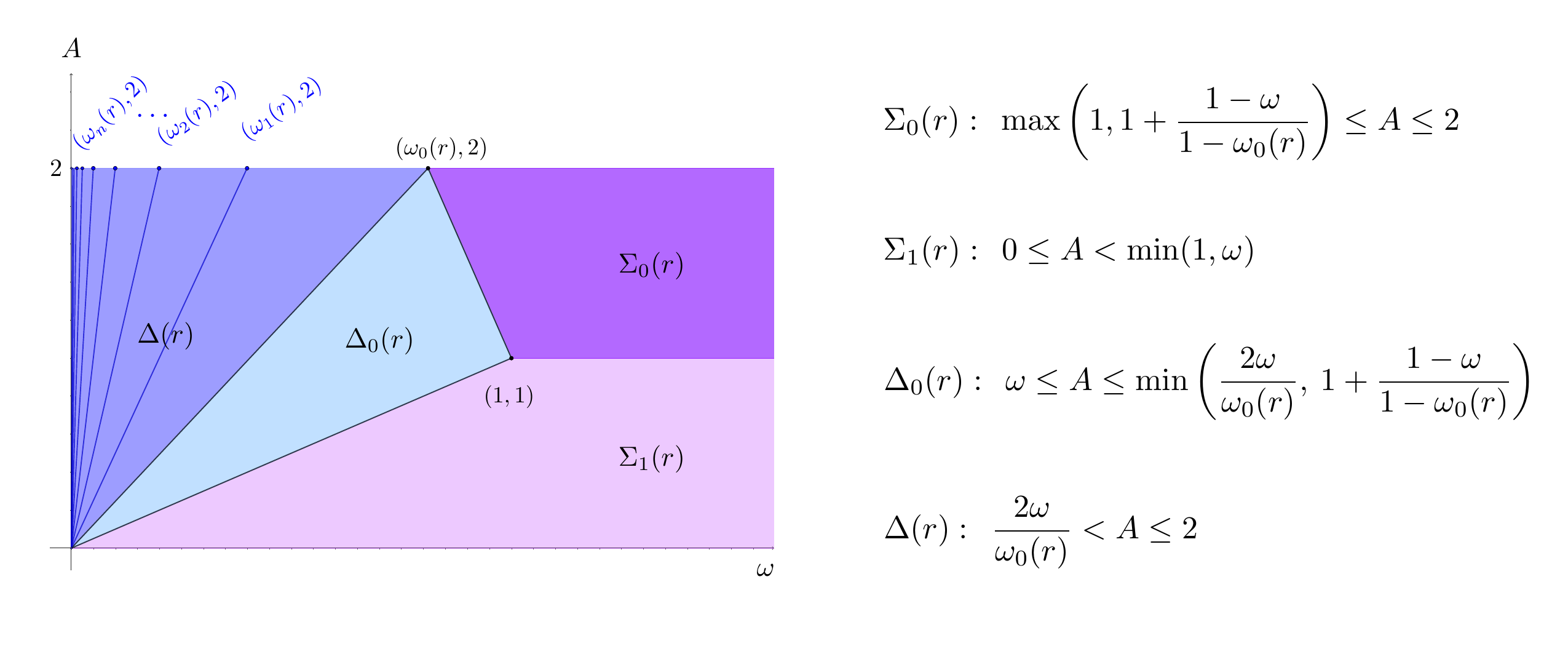}
\end{center}
\caption{Regions partitioning the $\Omega_\mbm$-domain.}
\label{fig:Mregions1}
\end{figure}

\begin{theorem}
        \label{t:ExactBellman}
    The Bellman function defined in \eqref{E:B-def1} is given by
        $$
        \mbb_r(x,A,\lambda)=\mbm_r\left(x \lambda^{-\sfrac{1}{r}},\: A\right).
        $$
\end{theorem}

\noindent We prove this in the next sections. In Section \ref{S:Bellman}, we lay out the standard Bellman function arguments, culminating with the ``Least Supersolution'' property of $\mbb_r$ -- which allows us to transition to an optimization problem. As explained in Section \ref{subS:supersolutions}, $\mbb_r$ can also be understood as the pointwise minimizer of a certain class of functions $\scb$ (which contains $\mbb_r$!). We then set out to construct such a minimizer $\bfb$ of the class $\scb$, in Section \ref{S:Construction}. 

Now, this constructive approach always runs the risk that we could have ``done better:'' as a trivial example, the $0$ function minimizes the collection $\scb$ but it is certainly not the best minimizer. So far we have $\bfb\leq\mbb_r$, which only makes $\bfb$ a \textit{candidate} for the true Bellman function. To prove that $\bfb$ is indeed $\mbb_r$, we show in Section \ref{S:ProofTrue} that $\bfb$ belongs to the collection $\scb$.

\subsection{Weak-type bounds} 
Now that we have a full description of $\mbm_r$, we return to \eqref{e:W11-withB} and questions about $\mca_{\alpha, r}$. We evaluate the quantity on the right hand side below:
    $$
    \sup_{\alpha \in \mfs_2(I)} \|\mca_{\alpha, r}\|_{(1,1)} = \sup_{A\in [0,2]} \:\sup_{\omega>0} \: \frac{\mbm_r(\omega, A)}{\omega}.
    $$
Recall first a simple result below.

\begin{lemma}
    \label{L:concave}
    If $f: [0,\infty) \rightarrow [0,\infty)$ is concave, then the function $f(x)/x$ (for $x>0$) is non-increasing. Consequently,
    $$
    \sup_{x>0} \: \frac{f(x)}{x} = \lim_{x\rightarrow 0+} \: \frac{f(x)}{x}.
    $$
\end{lemma}

\begin{proof}
    Let $0<x<y$. Since $f$ is concave and $f(0)\geq 0$, we let $\theta:=\frac{x}{y} \in (0,1)$ below:
        $$
       f(x) = f\big( (1-\theta)\cdot 0 + \theta \cdot y\big) \geq \theta f(y) = \frac{x}{y} f(y),
        $$
    which proves the statement.
\end{proof}

As discussed in Section \ref{S:ProofTrue}, the function $\mbm_r$ is concave. However, it is not smooth -- so this is perhaps handled quicker as follows.


\begin{figure}[h]
\begin{center}
\includegraphics[scale=0.4]{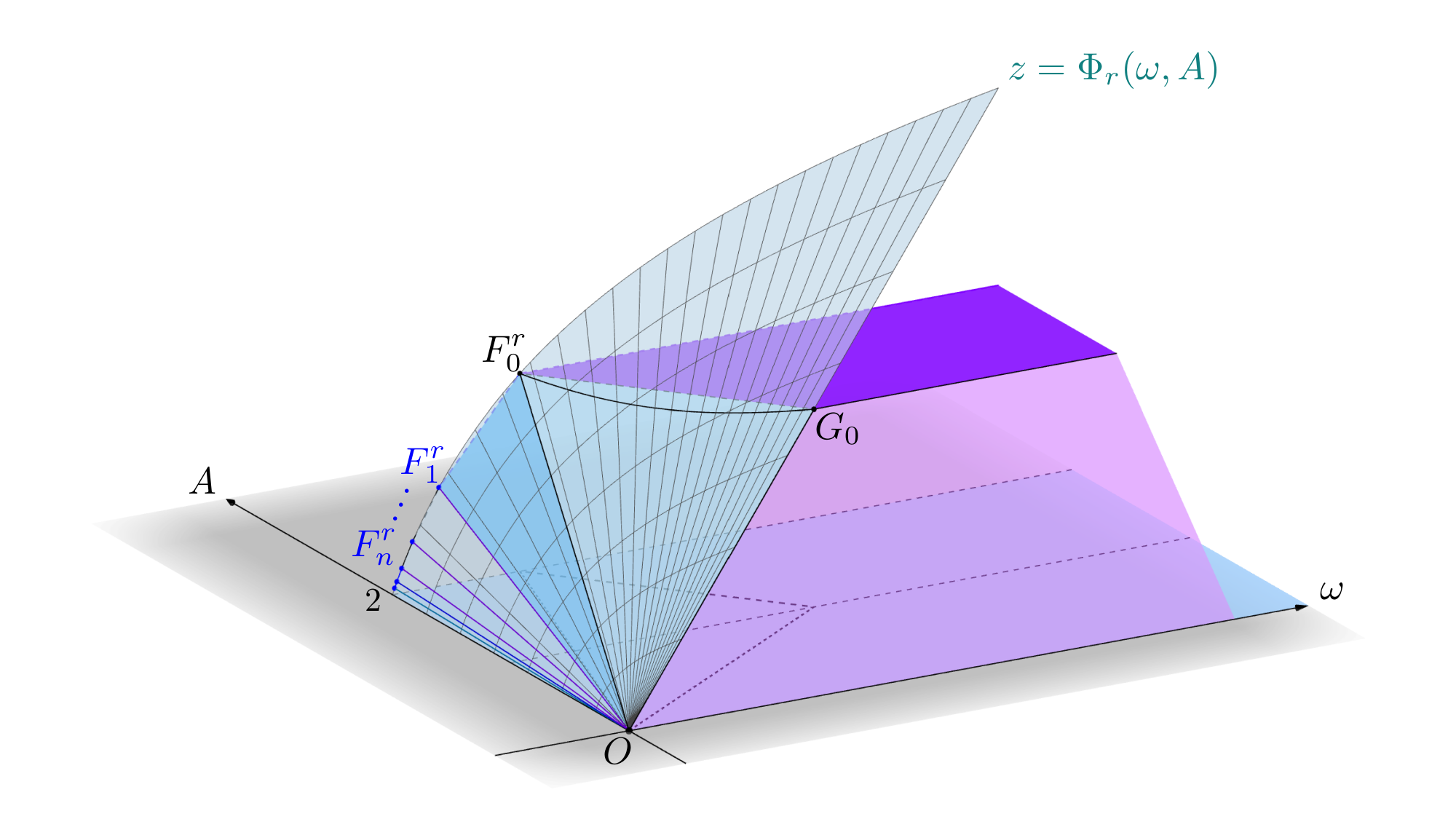}
\end{center}
\caption{Graph of $\mbm_r$ and the ``Bellman envelope'' surface $z=\Phi_r(\omega, A)$.}
\label{fig:MgraphPhi}
\end{figure}

Referring back to the graph in Figure \ref{fig:Mgraphr}, observe that the ``seams'' gluing the triangles together, i.e. the 
line segments $[OG_0]$ and $[O F_n^r]$; $n\geq 0$, all lie on the surface
$z=\Phi_r(\omega, A)$ (pictured in Figure \ref{fig:MgraphPhi}), where
    \begin{equation*}
        \Phi_r(\omega, A) := A\omega \cdot \left(
        \frac{2^{r+1}-1}{2^r \omega^r + (2^r-1)A^r}
        \right)^{\sfrac{1}{r}} ; \:\:\: 0\leq \omega \leq A \leq 2.
    \end{equation*}
This is a smooth concave function which coincides with the piecewise linear $\mbm_r$ along the edges $[OG_0]$, $[O F_n^r]$, so
    $$\Phi_r(\omega, A) \geq \mbm_r(\omega, A), \:\text{ for all } \: 0\leq \omega \leq A \leq 2.$$
Then, for every fixed $A$,
    $$
    \sup_{\omega>0} \: \frac{\mbm_r(\omega,A)}{\omega} = \lim_{\omega \rightarrow 0+} \: \frac{\mbm_r(\omega,A)}{\omega}
    \leq \lim_{\omega \rightarrow 0+} \: \frac{\Phi_r(\omega,A)}{\omega} = \left(
    \frac{2^{r+1}-1}{2^r-1}
    \right)^{\sfrac{1}{r}}.
    $$
But the limits are actually equal, since the functions coincide on the sequence (converging to $0$) $\{u_n\}_{n\geq 0}$, where $(u_n, A)$ are points on the line segments $[OF_n^r]$. So
    \begin{equation}
        \label{e:bound}
      \sup_{\alpha \in \mfs_2(I)} \|\mca_{\alpha, r}\|_{(1,1)} = 
      \sup_{A\in [0,2]} \: \sup_{\omega>0} \: \frac{\mbm_r(\omega, A)}{\omega} = C(r):= \left(
    \frac{2^{r+1}-1}{2^r-1}
    \right)^{\sfrac{1}{r}},
    \end{equation}
proving Proposition \ref{prop:bound}.

\begin{center}
  {\Large\ding{167}} 
\end{center}

Consider again the points 
$F_n^r = (\omega_n(r), \: 2, \: 2^{-n})$, $n\geq 0$, in Figure \ref{fig:Mgraphr}, which essentially ``pull the strings'' of the surface $z=\mbm_r(\omega, A)$ as $r \in (0,\infty)$ varies, and note that
    $$
    \lim_{r\rightarrow 0_+} \omega_n(r) = 0; \:\:\:\:
    \lim_{r\rightarrow \infty} \omega_n(r)= 2^{-n}.
    $$
Therefore the family of Bellman functions $\{\mbm_r\}_{r>0}$ has well-defined pointwise limits at the extremes (pictured at the opposite ends of Figure \ref{fig:acc}): as $r\rightarrow 0_+$, all but one of the triangles collapse onto the plane $\omega=0$,
    \begin{equation}
        \label{e:M0-def}
    \mbm_{0_+}(\omega, A) := \lim_{r\rightarrow 0_+} \mbm_r(\omega, A) = 
    \sup_{r>0} \mbm_r(\omega, A)
    =\min\left(1 \:, A,\: \frac{\omega+A}{2}\right),
    \end{equation}
and as $r\rightarrow \infty$, all triangles lay down on the plane $z=\omega$:
    \begin{equation}
        \label{e:Minf-def}
        \mbm_{\infty}(\omega, A) := \lim_{r\rightarrow \infty} \mbm_r(\omega, A)  
        = \inf_{r>0} \:\mbm_r(\omega, A)
        = \min(1, \: A, \: \omega).
    \end{equation}

\begin{figure}[h]
\begin{center}
\includegraphics[width=\linewidth]{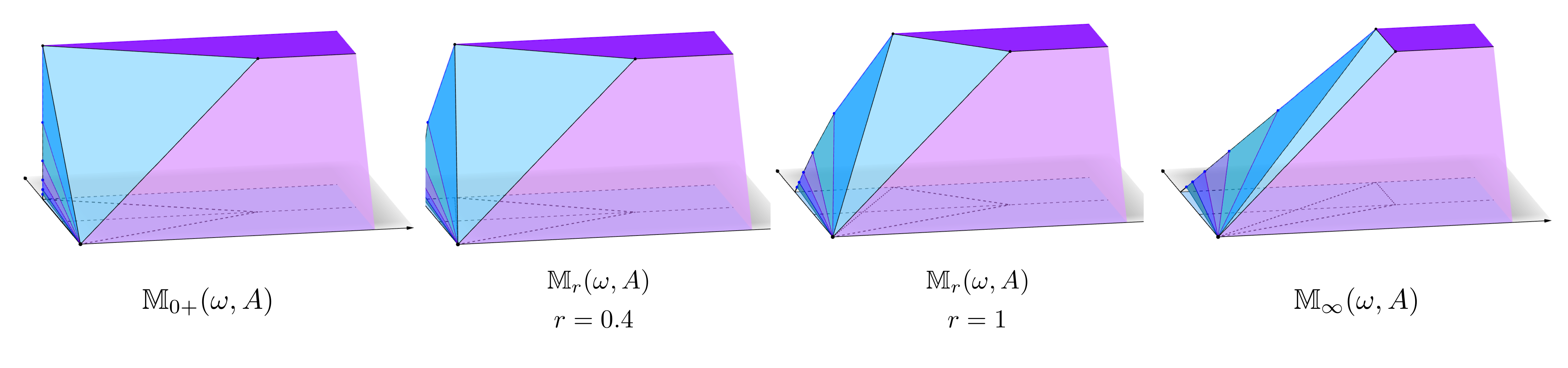}
\end{center}
\caption{An accordion of Bellman functions.}
\label{fig:acc}
\end{figure}


\subsection{Discussion} To best interpret these results, we should frame the conversation not in terms of $\mcd(I)$, but in terms of the \textit{sparse generational structure} imposed by $\alpha$ (see Figure \ref{fig:sparse} for an example). For any $J\in\mcd(I)$, we can define the collection of its \textit{$\alpha$-children}:
    $$
    \texttt{ch}_\alpha(J) := \{\text{maximal } K\in\mcd(I) \text{ such that } K\subsetneq J \text{ and } \alpha_K=1\},
    $$
where maximality is with respect to set inclusion. Note that, while any $J\in\mcd(I)$ has exactly two \textit{dyadic children}, $J$ can have anywhere from zero to countably many $\alpha$-children. The $\alpha$-generations are then
    $$
    \mcg_\alpha^0 := \{\text{maximal } K\in\mcd(I) \text{ such that } \alpha_K=1\}, \:\:\: \text{and}
    \:\:\: \mcg_\alpha^m := \bigcup_{K\in\mcg^{m-1}_\alpha} \: \texttt{ch}_{\alpha}(K) \text{ for } m\geq 1.
    $$
Each $\mcg_\alpha^m$ is a \textit{disjoint} collection of intervals in $\mcd(I)$, and we denote their union by $S_\alpha^m$. 

\begin{figure}[h]
\begin{center}
\includegraphics[width=\linewidth]{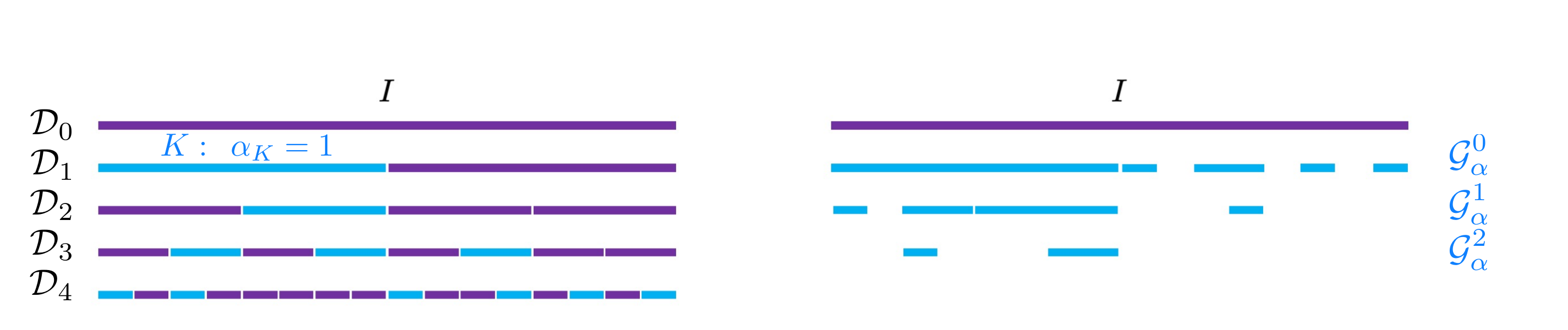}
\end{center}
\caption{A sparse collection $\alpha$, shown to the left spanning five dyadic generations within $\mcd(I)$, and shown to the right with its three sparse generations.}
\label{fig:sparse}
\end{figure}

\vspace{0.1in}

Observe that the support of $\mca_{\alpha,r}f$ is not just in $I$, but in $S_\alpha^0$. If $\alpha_I=1$ (i.e. the main interval $I$ is selected for $\scs_\alpha$), then $S_\alpha^0 = I$; otherwise, $S_\alpha^0$ can be a disjoint union which does not cover all of $I$.

Now, it is possible for $f$ to be, in a way, orthogonal to $\alpha$, i.e. if
    $
    f=0 \text{ a.e. on } S_\alpha^0 \subsetneq I.
    $
In that case, $\mca_{\alpha, r}f \equiv 0$ for all $r>0$. Assume going forward that this is not the case ($f \not\perp \alpha$). Equivalently, we assume there is at least one $K\in\mcg_\alpha^0$ such that $\La f\Ra_K\neq 0$.

\begin{center}
  {\Large\ding{167}} 
\end{center}

Looking at these operators $\mca_{\alpha,r}f$ across the $r$-spectrum, we find a delicate balance between the effect of $\scs_\alpha$'s geometry, versus the effect of the function $f$. 

\vspace{0.1in}


The power shift in $\alpha$ vs. $f$ occurs at $r=1$ (the classical sparse operator), which acts as a sort of ``equilibrium point'' where $\scs_\alpha$ and $f$ have balanced influence.  For $r<1$, the more important factor is the sparse geometry of $\scs_\alpha$, i.e. the ``amount of overlap'' between $\scs_\alpha$-intervals -- which is exactly what the Carleson condition keeps under control. For a function $f$, the overlap in $\alpha$ is much more costly for $r<1$ than it is for $r>1$.
This is a simple consequence of the classical inequality for sequences of non-negative reals:
    $$
    \text{For all } r\geq 1: \: 
    \big(\sum a_n^r\big)^{\sfrac{1}{r}} \leq \sum a_n \leq \big(\sum a_n^{\sfrac{1}{r}}\big)^r.
    $$
In context,
    $$
    \text{For all } r\geq 1: \: \left(\sum\limits_{K\in\scs_\alpha}\langle f\rangle_K^r\right)^{\sfrac{1}{r}} \leq \sum\limits_{K\in\scs_\alpha} \langle f\rangle_K \leq \left(\sum\limits_{K\in\scs_\alpha} \langle f\rangle_K^{\sfrac{1}{r}}\right)^r$$


\noindent So, it is not surprising that $C(r)$ blows up as $r \rightarrow 0+$: in this range, $\La f\Ra_I$ can become very small, while the level sets of $\mca_{\alpha, r}$ remain relatively balanced.

\vspace{0.1in}

The superpower of the sparseness condition is that it often serves as the next best thing to disjointness. This extreme situation, when $\scs_\alpha$ is a disjoint collection, occurs if and only if
    $
    \|\alpha\|_{\text{Car}} := \sup_{J\in\mcd(I)} \: A(\alpha; \: J) = 1,
    $
and in this case the $r$-spectrum remains unchanged for all $f$:
    $
    \mca_{\alpha, r} f(t) = \mca_{\alpha, 1} f(t), \text{ for all } r>0 \text{ and all } f\in L^1(I).
    $
As $\|\alpha\|_{\text{Car}}$ increases past $1$, the amount of overlap in $\alpha$ increases, with wildly different consequences for small versus large $r$.

\subsection{Extremes} To understand just how much $\alpha$ overpowers $f$ for small $r$, remark that
    $$
    \lim_{r\rightarrow 0_+} \mca_{\alpha, r}^r f(t) = \lim_{r\rightarrow 0_+} \:
    \sum_{\substack{K\in \scs_\alpha:  K \ni t}} \La f\Ra_K^r = 
    \sum_{K\in \scs_{\alpha}: \La f\Ra_K\neq 0} \unit_K(t),
    $$
so the averages of $f$ do not matter at all: the most $f$ can do is ``turn off' some intervals where it may be identically 0. For non-vanishing $f$, we always recover the full function
    $$
    \mca_\alpha \unit(t) := \sum_{K\in\scs_\alpha} \:\unit_K(t),
    $$
which is associated to any sparse collection. For every $t\in I$, this function \textit{counts} the number of $\scs_\alpha$-intervals which contain $t$, so $\mca_\alpha \unit$ takes values in $\mathbb{Z}_{\geq 0}$ (it may be $\infty$ for some $t$, but the Carleson condition ensures that the set of all such $t$ has measure zero).

We can take this further: let $f\geq 0$ in $L^1(I)$ with $\La f\Ra_I \neq 0$, and $\alpha \in \mfs_2(I)$ with $f \not\perp \alpha$, and consider, for $t \in I$,
    \begin{equation}
        \label{e:lim0}
        \mca_{\alpha, 0_+} f(t) := \lim_{r\rightarrow 0+} \:\mca_{\alpha, r}f(t)
        = \lim_{r\rightarrow 0+} \: \left(\sum_{\substack{K\in \scs_\alpha:  K \ni t}} \La f\Ra_K^r\right)^{\sfrac{1}{r}}.
    \end{equation}
As expected, this is very badly behaved, but it clarifies  just how costly any overlap in $\alpha$ is for small $r$. There are only two situations when \eqref{e:lim0} is finite: when the sum inside the parentheses is $0$, and when there is exactly one non-zero term in this sum. If $t\in I\setminus S_\alpha^0$, then $\mca_{\alpha, r}f(t)=0$ for all $r>0$, in which case the limit is $0$. So let $t\in S_\alpha^0$:
    \begin{itemize}
        \item  If there exists a $K\in\scs_\alpha$ in any lower sparse-generation $\mcg_\alpha^m$, $m\geq 1$, such that $t\in K$ \textit{ and } $\La f\Ra_K\neq 0$ (i.e. $f$ does not vanish on $K$), then the limit in \eqref{e:lim0} is $\infty$. 

        \item Otherwise, 
            $$
            \mca_{\alpha, r}f(t) = \La f\Ra_{K_0}, \text{ for all } r>0,
            $$
        where $K_0$ is the unique element in $\mcg_\alpha^0$ containing $t$. 
    \end{itemize}
For example, if $t\in S_\alpha^0 \setminus S_\alpha^1$, then the limit in \eqref{e:lim0} is $\La f\Ra_{K_0}$ across all functions $f$. Or, 
            $$
            \text{If } f \text{ vanishes on } S_\alpha^1, \text{ then }
            \mca_{\alpha, r}f(t) = \sum_{K_0 \in \mcg_\alpha^0} \La f\Ra_{K_0} \unit_{K_0}(t), \text{ for all } r>0,
            $$
        so in this case $\mca_{\alpha,0_+}f(t)$ is the well-defined sum above.

\vspace{0.1in}
In other words, any meaningful interaction between $[\mca\unit\geq 2]$ and $[f>0]$ immediately leads to a blowup. This can only be avoided by turning off one of these forces.
 
\vspace{0.1in}
This limit may be strange, but if the goal is to maximize level sets, $\infty$ is good. Let's try to make sense of the function in \eqref{e:Minf-def}
$$
\sup_{f,\alpha} \: \frac{1}{|I|}|\{t\in I: \mca_{\alpha,0_+}f(t) \geq 1\}| = 
\mbm_{0_+}(\omega, A) 
    =\min\left(1 \:, A,\: \frac{\omega+A}{2}\right),
$$
where supremum is over all $f\geq 0$ on $I$ with $\La f\Ra_I = \omega>0$, and all $\alpha \in \mfs_2(I)$ with $A(\alpha; \: I)=A$.
Since even one degree of overlap is enough for $\mca_{\alpha,0_+}f$ to blow up, it is enough to consider $\alpha$'s with only two sparse generations. 
Given $f$ and $\alpha$, what is the set $\{t\in I: \mca_{\alpha,0_+}f(t) \geq 1\}$?
    \begin{itemize}
        \item Any second $\alpha$-generation $K_1 \in \mcg_\alpha^1$ is automatically captured for the level set, as long as $\La f\Ra_{K_1} \neq 0$. So all $f$ has to do in order to capture $S_\alpha^1$ is not vanish there.
        \item Any remaining sets $K_0\setminus S_\alpha^1$ with $K_0\in\mcg_\alpha^0$ are either captured entirely -- if $\La f\Ra_{K_0}\geq 1$, or not at all if $\La f\Ra_{K_0}<1$.
    \end{itemize}
This explains the counterintuitive non-zero value of $\mbm_{0_+}$ at $\omega=0$ (which means $f\equiv 0$ on $I$!):
    $$
    \mbm_{0_+}(0,A) = \sfrac{A}{2}.
    $$
Roughly speaking, $\alpha$ is strong enough to capture all of $S_\alpha^1$ with barely any $f$ at all -- and the most $|S_\alpha^1|/|I|$ can be is $\sfrac{A}{2}$. So take $\alpha$ such that $|S_\alpha^0|/|I|= |S_\alpha^1|/|I| = \sfrac{A}{2}$. 
For any $\epsilon>0$, we can take a function $f=\sfrac{2\epsilon}{A} \unit_{S_\alpha^0}$. 
 Then $\La f\Ra_I=\epsilon$, and $f>0$  captures all of $S_\alpha^1$.

\vspace{0.1in}
Why is $\mbm_{0_+}(\omega,2)$ always $1$? Because for $A=2$ we can take $\scs_\alpha := \{I, \: I_-, \: I_+\}$, and then for any $\omega>0$ we can ensure $f$ does not vanish on either one of $I_-$ and $I_+$.

\vspace{0.1in}

For $\omega>0$ and $A<2$, it becomes a balancing act between how much of $f$ can we take to be  greater or equal to $1$ on $S_\alpha^0 \setminus S_\alpha^1$, and have enough left over to capture $S_\alpha^1$? This is a Lagrange multiplier optimization problem with solution precisely $\min(1, A, \tfrac{\omega+A}{2})$.


\begin{center}
  {\Large\ding{167}} 
\end{center}

\noindent \textbf{The limit as $r\rightarrow \infty$}. For $f\geq 0$ in $L^1(I)$ and $\alpha \in \mfs_2(I)$, the sequence $\{\mca_{\alpha, r}f\}_{r>0}$ converges pointwise almost everywhere to the maximal function operator adapted to $\alpha$:
    $$
    \mca_{\alpha, \infty}f(t) := \lim_{r\rightarrow\infty} \mca_{\alpha, r} f(t) = \sup_{K\in\scs_\alpha} \La f\Ra_K \unit_K(t) \text{ for a.a. } t\in I.
    $$
The Bellman function for level sets of this operator is given exactly by the function in \eqref{e:Minf-def},
    $$
    \sup_{f, \alpha} \frac{1}{|I|}
    |\{t\in I: \mca_{\alpha, \infty} f(t) \geq 1\}|=
    \mbm_{\infty}(\omega, A) = \min(1, \: A, \: \omega),
    $$
with supremum over all $f\geq 0$ on $I$ with $\La f\Ra_I = \omega$, and all $\alpha \in \mfs_2(I)$ with $A(\alpha; \: I)=A$. We can see this directly, without appealing to Bellman arguments. Note that
the level set $\{t\in I: \mca_{\alpha, \infty} f(t) \geq 1\}$ is the disjoint union of the collection of maximal $K\in\scs_\alpha$ (if any) such that $\La f\Ra_K\geq 1$. 
    \begin{itemize}
        \item If $A\geq 1$ and $\omega\geq 1$, we can always take $\alpha$ with $\alpha_I=1$, and $f\equiv \omega$ on $I$, which wins us all of $I$ for the level set.

        \item If $A<1$ and $A<\omega$, the most we can expect is $|S_\alpha^0|/|I| \leq A <1$. We can attain the maximum $A$: take $\alpha$ to be a disjoint collection (only one sparse generation), and take $f$ to be $1$ on $S_\alpha^0$ and $\frac{\omega-A}{1-A} \geq 0$ on the rest of $I$. Then $\La f\Ra_I = \omega$ and $\La f\Ra_K= 1$ for all $K\in\mcg_\alpha^0$, so $\mca_{\alpha, \infty}f \geq 1$ on all of $S_\alpha^0$.

        \item Finally, if $\omega<1$ and $\omega<A$, note that we can only take $f\geq 1$ on a set of length at most $\omega<A$. 
    \end{itemize}


\subsection{Weak-type $(p,p)$ bounds for $\mcq_{\alpha, p}$ operators}
\label{sS:power-mean}

Let $p\geq 1$. By H\"{o}lder's inequality 
    \begin{equation}
        \label{e:comparison}
        \mca_{\alpha, p} f \leq \left(\mca_{\alpha, 1}(|f|^p)\right)^{\sfrac{1}{p}}
        \leq \mcq_{\alpha, p}f,
    \end{equation}
with all operators coinciding at $p=1$. For $\alpha\in \mfs_2(I)$:
    $$
    \|\mcq_{\alpha, p}\|_{(p,p)} := \|\mcq_{\alpha, p} : L^p(I) \rightarrow L^{p,\infty}(I)\| =
    \sup_{f, \lambda} \left(\frac{\lambda^p}{\La f^p\Ra_I} \frac{1}{|I|} |\{t\in I: \mcq_{\alpha, p}f(t) \geq \lambda\}|\right)^{\sfrac{1}{p}}
    $$
where the supremum is over all $f\geq 0$ in $L^p(I)$, and all $\lambda>0$. But observe that
    $$
    \mcq_{\alpha, p}f = \left(\mca_{\alpha, \sfrac{1}{p}}(f^p)\right)^{\sfrac{1}{p}},
    $$
so
    $$
    \sup_{\alpha\in\mfs_2(I)} \: \|\mcq_{\alpha, p}\|_{(p,p)} = \sup_{\substack{y>0, \: \lambda>0;\\ A\in [0,2]}} \left(\frac{\lambda^p}{y} \mbb_{\sfrac{1}{p}} (y, A, \lambda) \right)^{\sfrac{1}{p}}
    = \sup_{A\in [0,2]} \: \sup_{\omega>0} \left(\frac{\mbm_{\sfrac{1}{p}}(\omega, A)}{\omega}\right)^{\sfrac{1}{p}} = C(\sfrac{1}{p})^{\sfrac{1}{p}},
    $$
proving Corollary \ref{cor:power-mean}.

\section{Bellman Function Properties}
\label{S:Bellman}

Let $r>0$, and recall our Bellman function:
\begin{equation} \label{E:B-def}
\mbb_r (x,A,\lambda) := \sup_{f, \alpha} \frac{1}{|I|} |\{t\in I: \mca_{\alpha, r}^r f(t) \geq\lambda\}|,
\end{equation} 
where
    $$
    \mca^r_{\alpha, r}f(t) = \sum_{J\in\mcd(I)} \alpha_J \: \langle |f|\rangle_J^r \: \unit_J(t),
    $$
and supremum is over all
\begin{itemize}
\item non-negative functions $f\geq 0$ in $L^1(I)$ with fixed average 
    $    \langle f \rangle_I = x;   $
\item sequences $\alpha\in\mfs_2(I)$ with 
    $    A(\alpha; I) = A.    $
\end{itemize}
We say any such pair $(f,\alpha)$ is an \textit{admissible pair} for $\mbb_r(x,A,\lambda)$. For the rest of this paper, we assume $r>0$ is fixed and suppress the $r$-subscript.

\vspace{0.1in}
We usually define the domain of a Bellman function to be the set of all points where $\mbb>-\infty$ (i.e. we are not taking supremum over the empty set). In this case, 
the domain is
    $$
    \Omega_\mbb:=[0,\infty) \times [0,2] \times \mbr.
    $$
It is obvious from the definition that
    $$
    0 \leq \mbb(x,A,\lambda)\leq 1.
    $$
The next theorem lists the basic properties of $\mbb$. 

\begin{theorem}
    \label{T:Bellman-Props}
    The function $\mbb:\Omega_\mbb\rightarrow [0,1]$ defined above satisfies the following:

    \begin{enumerate}[label=\textnormal{(\roman*)}, itemsep=0.6em, leftmargin=2.2em]
        \item \textbf{Interval Independence:} $\mbb$ is independent of the choice of main interval $I$.

        \item \textbf{Obstacle Condition:} For all $\lambda \leq 0$,
        \begin{equation}
            \label{e:mbb-OC}
            \mbb(x,A,\lambda) = 1.
        \end{equation}

        \item \textbf{Initial Conditions:} For all $\lambda > 0$,
        \[
        \mbb(0,A,\lambda) = 0 \quad \text{and} \quad \mbb(x,0,\lambda) = 0.
        \]

        \item \textbf{Homogeneity:} For all $c > 0$,
        \begin{equation}
            \label{e:mbb-Hom}
            \mbb(x,A,\lambda) = \mbb(cx, A, c^r \lambda).
        \end{equation}

        \item \textbf{Main Inequality:} For all points in the domain,
        \begin{equation}
            \label{e:mbb-MI}
            \mbb(x, A+\gamma, \lambda + \gamma \cdot x^r) \geq \frac{1}{2}\bigg(\mbb(x_1,A_1,\lambda) + \mbb(x_2,A_2,\lambda)\bigg),
        \end{equation}
        where
        \[
        x = \La x_i \Ra := \frac{x_1 + x_2}{2}, \quad
        A = \La A_i \Ra := \frac{A_1 + A_2}{2}, \quad
        \gamma \in \{0,1\}, \quad
        A + \gamma \leq 2.
        \]
    \end{enumerate}
\end{theorem}

\vspace{10pt}

\begin{proof}

    {\noindent\em (i)} If $f\in L^1([a_0,b_0))$, define $g\in L^1([a_1,b_1))$ as $g(x):=f(a_0+\frac{b_0-a_0}{b_1-a_1}(x-a_1))$. Then clearly 
    $\langle g\rangle_{[a_1,b_1)} = \langle f\rangle_{[a_0,b_0)} $, and
    \begin{equation*}
        \frac{1}{b_0-a_0}|\{t\in [a_0,b_0): \mca_{\alpha,r}^r f(t) \geq\lambda\}|=\frac{1}{b_1-a_1}|\{t\in [a_1,b_1): \mca_{\alpha,r}^r g(t) \geq\lambda\}|.
    \end{equation*}

    \vspace{0.1in}

    {\noindent\em (ii)} This is an immediate consequence of the fact that $f\geq 0$.

    \vspace{0.1in}

    {\noindent\em (iii)} The first equality follows from the realization that a nonnegative function with average $\langle f\rangle = 0$ must satisfy $f\equiv0$ a.e. The second equality is obvious: if $A=0$, then $\alpha_J=0$ for all $J\in\mathcal{D}(I)$ and $\mca_{\alpha,r}^rf(t)=0$ for all $t$.

    \vspace{0.1in}

    {\noindent\em (iv)} Let $(f, \: \alpha)$ be an admissible pair for $\mbb(x,A,\lambda)$. Then, for every $c>0$, the function $c f$ satisfies $\langle cf\rangle_I = c\langle f\rangle_I = cx$. Therefore $(cf, \: \alpha)$ is an admissible pair for $\mathbb{B}(cx,A,c^r\lambda)$, so
        $$
        \mathbb{B}(cx,A,c^r\lambda) \geq \frac{1}{|I|} \left|\left\{ t\in I: \mca_{\alpha,r}^r (cf) (t) \geq c^r\lambda \right\}\right| = 
        \frac{1}{|I|}\left|\left\{ t\in I: \mca_{\alpha,r}^r (f) (t) \geq \lambda \right\}\right|.
        $$
    Since this holds for all pairs $(f,\alpha)$ admissible for $\mbb(x,A,\lambda)$, we have
        $$
        \mbb(x,A,\lambda) \leq \mbb(cx, A, c^r \lambda), \:\text{ for all } c>0.
        $$
    We can replace $c$ with $\sfrac{1}{c}$ above and obtain the equality in \eqref{e:mbb-Hom}.
    
       \vspace{0.1in}

    {\noindent\em (v)} Let $\epsilon>0$ and $(x_i, A_i, \lambda)$, $i=1,2$ be two points in the domain of $\mbb$. Let $I$ be an interval, and denote its left and right halves by $I_1=I_-$ and $I_2=I_+$, respectively. For each $i=1,2$, let
        $(f_i, \: \alpha_i)$ be an admissible pair for $\mbb(x_i, A_i, \lambda)$, where
            \begin{align*}
                & f_i \in L^1(I_i); \: f_i \geq 0; \: \La f_i\Ra_{I_i}=x_i\\
                & \alpha_i \in \mfs_2(I_i); \: A_i=A(\alpha_i;I_i).
            \end{align*}
    Moreover, we can assume $(f_i, \: \alpha_i)$ were chosen in such a way that they ``almost give the supremum:''
        \begin{equation}
            \label{e:MIbla}
        \frac{1}{|I_i|}|\{t\in I_i: \mathcal{A}_{\alpha_i, r}^rf_i\geq\lambda\}|\geq \mathbb{B}(x_i,A_i,\lambda) - \epsilon.
        \end{equation}
    Now, concatenate $f_1$ and $f_2$ to obtain:
    \begin{equation*}
        f(t)=f_1(t)\unit_{I_1}(t) + f_2(t)\unit_{I_2}(t),
    \end{equation*}
   a function in $L^1(I)$ with 
    $$
    \La f\Ra_I = x := \La x_i \Ra.
    $$

When concatenating the sparse collections $\alpha_1$ and $\alpha_2$, to obtain a new sparse collection $\alpha$ \textit{adapted to $I$}, we have two choices for assigning $\alpha_I$: either $\alpha_I :=0 $, or $\alpha_I := 1$ (in  \eqref{e:mbb-MI}, the parameter $\gamma$ plays the role of $\alpha_I$). Then
    \begin{equation*}       
        \mca_{\alpha,r}^rf(t)=\alpha_Ix^r\unit(t) + \sum\limits_{J\in\mathcal{D}(I_1)}\alpha_J\langle |f|\rangle^r\unit(t)+\sum\limits_{J\in\mathcal{D}(I_2)}\alpha_J\langle |f|\rangle^r\unit(t).
    \end{equation*}
The new collection $\alpha$ has total mass
    $$
    A(\alpha,I)=A+\alpha_I, \text{ where } A := \La A_i\Ra.
    $$
Assuming $A+\alpha_I \leq 2$, the pair $(f, \alpha)$ is then admissible for $\mbb(x, A+\gamma, \lambda + \gamma x^r)$, where $\gamma := \alpha_I$. Then
    \begin{align*}
    \mbb(x, A+\gamma, \lambda + \gamma x^r)  & \geq
    \frac{1}{|I|} \left|\left\{t\in I: \mca_{\alpha,r}^rf(t) \geq \lambda+\gamma x^r\right\}\right|
     = \frac{1}{2} \sum_{i=1}^2 \frac{1}{|I_i|} \left|\left\{
    t \in I_i: \mca_{\alpha_i, r}^r f_i(t) \geq \lambda
    \right\}\right|
    \end{align*}
By our assumption in \eqref{e:MIbla}, after taking $\epsilon \rightarrow 0$, we obtain exactly \eqref{e:mbb-MI}.
\end{proof}

\subsection{The family of supersolutions}
\label{subS:supersolutions}

Let $\mathscr{B}$ denote the collection of all functions  $$
    B(x,A,\lambda) : \Omega_\mbb \rightarrow [0,1]
    $$
 which satisfy the Obstacle Condition:
        \begin{equation}
            \label{e:B-OC}
            B(x,A,\lambda)=1, \: \text{ for all } \lambda\leq 0,
        \end{equation}
 and the Main Inequality:           
 \begin{equation}
                \label{e:B-MI}
                B(x, \: A+\gamma,\: \lambda+\gamma \cdot x^r) \geq \frac{1}{2}\bigg(B(x_1,A_1,\lambda) + B(x_2,A_2,\lambda)\bigg),
            \end{equation}
for all points in the domain with
            $
            x=\La x_i\Ra; A=\La A_i\Ra; \gamma \in \{0,1\}; A+\gamma \leq 2.
            $

\begin{theorem}[Least Supersolution Property of $\mbb$]
    \label{T:LS}
    For all points $(x,A,\lambda)\in\Omega_\mbb$, we have:
    $$
    \mbb(x,A,\lambda) \leq B(x,A,\lambda),
    $$
    for all $B\in\scb$.
\end{theorem}

\begin{proof}

Fix a supersolution function $B\in \scb$ and a point $(x,A,\lambda)$ in the domain. We will show that
    \begin{equation}
        \label{E:LSstar}
    \frac{1}{|I|} \left|\left\{t \in I: \mca_{\alpha,r}^r f(t) \geq \lambda\right\}\right|
    \leq B(x,A,\lambda)
    \end{equation}
holds for all pairs $(f,\alpha)$ admissible for $\mbb(x,A,\lambda)$. 

By a standard limiting argument, it is enough to show that \eqref{E:LSstar} holds for sequences $\alpha$ with finitely many non-zero terms. Assume there is some dyadic level $N\in\mbn$ such that
    $$
    \alpha_J =0, \text{ for all } J \in \mcd(I) \text{ with } \frac{|J|}{|I|} \leq \frac{1}{2^N}.
    $$
For such $\alpha$, the quantity $\mca_{\alpha,r}^r f$ is constant on each of these ``terminal'' intervals $J\in\mcd_N(I)$: 
    $$
    \mca_{\alpha, r}^r f(t) = \sum_{K \in \mcd(I): \: K \supsetneq J} \alpha_K \La f\Ra_K^r, 
    \text{ for all } t \in J, \: J \in \mcd_N(I).
    $$
To every $J\in\mcd(I)$ we associate the \textit{Bellman point} 
    $$
    p_J : = (x_J, A_J, \lambda_J),
    $$
where 
    $$
    x_J := \La f\Ra_J; \:\:\:\:\:\:\:
    A_J := A(\alpha; \: J); \:\:\:\:\:\:\:
    \lambda_J := \lambda - \sum_{K \in \mcd(I): \: K \supsetneq J} \alpha_K \La f\Ra_K^r.
    $$
Remark that these always satisfy
    $$
    x_J = \frac{x_{J_-} + x_{J_+}}{2}; \:\:\:
    A_J = \frac{A_{J_-}+A_{J_+}}{2} + \alpha_J; \:\:\:
    \lambda_{J_-} = \lambda_{J_{+}} = \lambda_J - \alpha_J x_J^r,
    $$
for all $J\in\mcd(I)$.

Begin with $p_I=(x_I, A_I, \lambda_I) = (x,A,\lambda)$, and apply the Main Inequality \eqref{e:B-MI} recursively for $B$:
    \begin{align*}
    B(x,A,\lambda) = B(p_I) & \geq 
    \frac{1}{2}\sum_{i=1}^2 \left(B(p_{I_-}) + B(p_{I_+})\right) 
    = \frac{1}{2} \sum_{J\in\mcd_1(I)} B(p_J)\\
    & \geq \frac{1}{2} \sum_{J \in \mcd_1(I)} \frac{1}{2}
    \left( B(p_{J_-}) + B(p_{J_+}) \right) = 
    \frac{1}{2^2} \sum_{J \in \mcd_2(I)} B(p_J)\\
    & \vdots\\
    & \geq \frac{1}{2^{N}} \sum_{J \in \mcd_{N}(I)} B(x_J, A_J, \lambda_J).
    \end{align*}
Now, observe that the level set $\{t\in I: \mca_{\alpha, r}^r f(t) \geq \lambda\}$ is a disjoint union of intervals $J\in\mcd_N(I)$, namely those for which $\lambda_J \leq 0$. 
Finally, by non-negativity of $B$ and the Obstacle Condition:
    $$
    B(x,A,\lambda) \geq \frac{1}{2^{N}} \sum_{J \in \mcd_{N}(I): \: \lambda_J \leq 0} B(x_J, A_J, \lambda_J) 
    = \frac{1}{2^{N}} \#\{J\in\mcd_N(I): \: \lambda_J \leq 0\},
    $$
which proves \eqref{E:LSstar}.

\end{proof}

\subsection{Transition to a minimization problem} 
Since $\mbb$ itself belongs to the collection $\scb$ (which is closed under minimization), Theorem \ref{T:LS} says that finding $\mbb$ is the same as finding the minimal function in $\scb$:
    $$
    \mbb(x,A,\lambda)= \min_{B\in \scb} B(x,A,\lambda).
    $$

Our strategy going forward will be to construct a ``candidate'' function $\bfb(x,A,\lambda)$, in such a way that
    $$
    \bfb \leq B, \text{ for all } B\in \scb.
    $$
We do this by propagating data through the domain using the Main Inequality, starting with
the initial data provided by the Obstacle.
This is detailed in Section \ref{S:Construction}.

Now, just because $\bfb$ is \textit{a} minimizer for the collection $\scb$, does not mean it is ``the best'' minimizer.
To see that our candidate $\bfb$ is indeed $\mbb$, we prove in Section \ref{S:ProofTrue} that $\bfb$ belongs to $\scb$.

\vspace{0.8em}
\begin{center}
  {\Large\ding{167}} 
\end{center}
\vspace{0.8em}

Before we proceed, let us note that \eqref{e:B-MI} can be simplified as follows:

\begin{lemma}
    A function $B$ satisfies \eqref{e:B-MI} if and only if $B$ satisfies the following two conditions:
     \begin{enumerate}[label=\textnormal{(\roman*)}, itemsep=0.6em, leftmargin=2.2em]
        \item \textbf{Concavity in the first two variables:}  
        \begin{equation}
        \label{e:B-C}
        B(x,A,\lambda) \geq \frac{1}{2}\sum_{i=1}^2 B(x_i, A_i, \lambda), \text{ for all }
        x=\La x_i\Ra; \: A=\La A_i\Ra.
        \end{equation}

    \item \textbf{Jump Inequality:}
        \begin{equation}
        \label{e:B-J}
        B(x,A+1,\lambda+x^r) \geq B(x,A,\lambda), \text{ for all } A\in [0,1].
    \end{equation}
    \end{enumerate}
\end{lemma}

\begin{proof}
    First suppose a function $B$ satisfies \eqref{e:B-MI}. Let $\gamma=0$ to obtain condition \eqref{e:B-C}, and let $\gamma=1$, $A=A_1=A_2$, and $x=x_1=x_2$ to obtain condition \eqref{e:B-J}.
    
    Conversely, suppose $B$ satisfies \eqref{e:B-C} and \eqref{e:B-J}, and let $\gamma\in\{0,1\}$. For $\gamma=0$, \eqref{e:B-MI} is precisely \eqref{e:B-C}. To obtain \eqref{e:B-MI} for the $\gamma=1$ case, we apply both \eqref{e:B-C} and \eqref{e:B-J}:
    $$
    B(x,A+1,\lambda + x^r) \geq B(x,A,\lambda) \geq \frac{1}{2}\sum\limits_{i=1}^2B(x_i,A_i,\lambda).
    $$
\end{proof}

\begin{remark}
    Since our functions are bounded, we can immediately upgrade midpoint concavity in \eqref{e:B-C} to full concavity (see page 12 in \cite{MR3363413}).
\end{remark}

\section{Constructing the Candidate $\bfb$}
\label{S:Construction}

Let any function $B \in \scb$. We will construct a function $\bfb$ such that $\bfb\leq B$ throughout the domain.
Now, recall that the ``true'' Bellman function $\mbb$ satisfies the homogeneity property \eqref{e:mbb-Hom}; since $\mbb$ is itself in $\scb$, there is no loss of generality if we restrict our search to functions $B\in\scb$ which also share this property. So suppose further that $B$ satisfies
            \begin{equation}
            \label{e:Hom-B}
            B(x,A,\lambda)=B(cx, A, c^r\lambda),
            \end{equation}
for all $c>0$.

%

\subsection{Jump the Obstacle} Let $A=1$ in the Jump Inequality \eqref{e:B-J}:
	$$
	B(x,1,\lambda) \geq B(x,0,\lambda-x^r).
	$$
If $\lambda-x^r\leq 0$, the Obstacle Condition says the right hand side above is 1, so:
	\begin{equation}
	\label{e:JumpO1-B}
	B(x,1,\lambda)=1, \text{ for all } \lambda \leq x^r.
	\end{equation}
Then, we can jump again: setting $A=2$ in \eqref{e:B-J} gives that
	$
	B(x,2,\lambda) \geq B(x,1,\lambda-x^r).
	$
From \eqref{e:JumpO1-B}, the right hand side above is 1 whenever $\lambda-x^r \leq x^r$, so
	\begin{equation}
	\label{e:JumpO2-B}
	B(x,2,\lambda)=1, \text{ for all } \lambda \leq 2x^r.
	\end{equation}

\subsection{Reduce to a two-variable problem} Now that we are done with negative values of $\lambda$ (needed to propagate the data from the Obstacle Condition), we can bring in homogeneity. Specifically, for $\lambda>0$, we can let $c=\lambda^{-1/r}$ in \eqref{e:Hom-B} and obtain 
	$
	B(x,A,\lambda)= B\big(x \lambda^{-\sfrac{1}{r}}, \: A, \: 1\big).
	$
Then we write:
	$$
	B(x,A,\lambda) = M(\omega, A), \text{ where } \omega:= x \lambda^{-\sfrac{1}{r}}.
	$$
	
Our new collection of interest, $\scm$, consists then of all functions $M(\omega, A)$ defined via
    $$
    M(\omega, A) := B(\omega, A, 1), \: \text{ for some } B\in\scb.
    $$
A straightforward translation of $\scb$-properties into $\scm$-language yields the following characterization of $\scm$.

\begin{lemma}
    The collection $\scm$ consists of all functions
        $
	M(\omega, A) : \Omega_\mbm \rightarrow [0,1], 
	$
    where
        $$
        \Omega_\mbm := [0,\infty) \times [0,2],
        $$
    which satisfy:

    \begin{enumerate}[label=\textnormal{(\roman*)}, itemsep=0.6em, leftmargin=2.2em]
        \item \textbf{Concavity:} $M$ is concave.
        \item \textbf{Jump Inequality}: For all $\omega \geq 0$ and $A\in [0,1]$:
           \begin{equation}
		\label{e:M-J}
		M\big(\varphi(\omega), A+1\big) \geq M(\omega, A), 
		\end{equation}
	where 
        	\begin{equation}
		    \label{e:varphi-def}
            \varphi_r : [0,\infty) \rightarrow [0,1); \:\:\:\:
            \varphi_r(\omega):=\frac{\omega}{(1+\omega^r)^{\sfrac{1}{r}}}.
		\end{equation}
		
        \item \textbf{Obstacle}:
            \begin{equation}
                \label{e:OC-M}
                M(\omega, 1)=1, \text{ for all } \omega\geq 1, \:\text{ and } \: 
                M(\omega, 2) = 1, \text{ for all } \omega \geq 2^{-\sfrac{1}{r}}.
            \end{equation}
    \end{enumerate}
\end{lemma}

\vspace{0.1in}
\noindent Our goal is now to find a function
	$
	\bfm(\omega, A)=:\bfb(\omega, A, 1)
	$
such that $\bfm\leq M$ for all $M\in\scm$.

\begin{figure}[h]
\begin{center}
\includegraphics[scale=0.25]{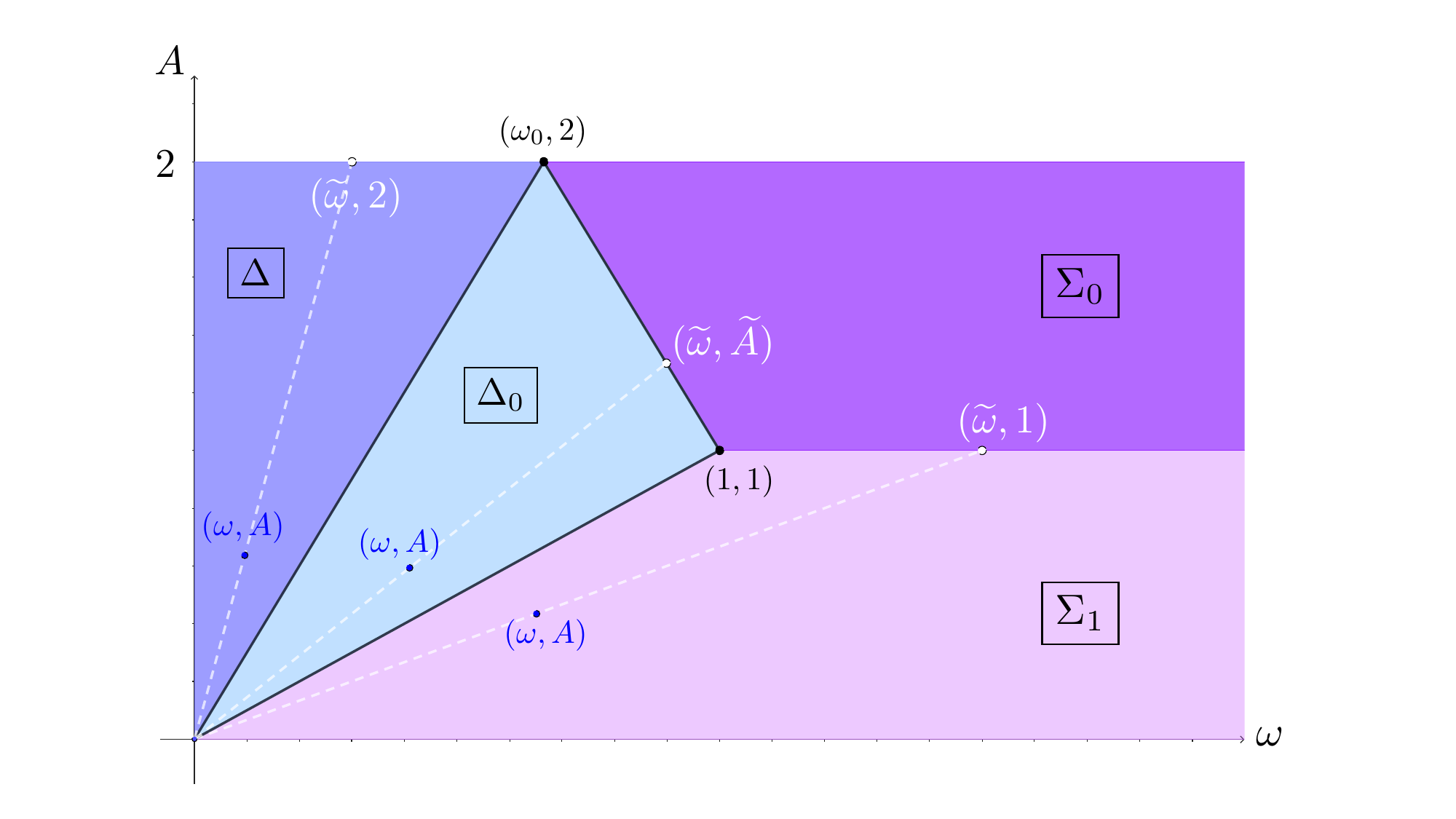}
\end{center}
\caption{Regions in the $M$-domain.}
\label{fig:Mregions}
\end{figure}

The first thing we notice is that $\scm$ functions attain their maximum along the two horizontal rays originating at 
$P_1 = (1,1)$, and at $P_2 = (2^{-\sfrac{1}{r}}, \: 2)$ -- see Figure \ref{fig:Mregions}. But these functions are also \textit{concave}, so  any $M\in\scm$ must be equal to $1$ in the convex hull of these rays. We will denote this (infinite) trapezoidal region by 
    $$
    \Sigma_0 := \left\{(\omega, A)\in\Omega_\mbm: \: \max\left(1, 1+ \frac{1-\omega}{1-\omega_0}\right) \leq A \leq 2\right\},
    $$
where the number
    $$
    \omega_0:=2^{-\sfrac{1}{r}}
    $$
will be fixed going forward. We have
    \begin{equation}
        \label{e:newInfo-0}
    M(\omega, A)=1, \text{ for all } (\omega, A)\in \Sigma_0.
    \end{equation}

 Next, we propagate the new information in \eqref{e:newInfo-0} through the rest of the $M$-domain, using concavity. The shape of $\Sigma_0$ forces us to divide the domain into three more regions: first, another infinite trapezoid 
    $$
    \Sigma_1 := \left\{ (\omega, A) \: :\: 0\leq A < \min(1,\omega)\right\},
    $$
and a triangular region
    $$
    \Delta_0 := \left\{ (\omega, A) \: :\: \omega\leq A\leq \min\bigg(\frac{2\omega}{\omega_0}, \: 1+\frac{1-\omega}{1-\omega_0}\bigg)\right\}.
    $$
In both these regions, we will define  $\bfm$ directly from the values along the boundary of $\Sigma_0$.

This leaves us with the second triangular region
    $$
    \Delta := \left\{ (\omega, A) \: :\: \frac{2\omega}{\omega_0} < A \leq 2 \right\},
    $$
where the main task is to construct $\bfm$ along the remaining $A=2$ boundary.

\vspace{0.1in}

Suppose we have a point $(\omega, A) \in \Sigma_1$.  The ray through the origin and $(\omega, A)$ intersects $A=1$ in the point $(\widetilde{\omega}, 1)$, where $\widetilde{\omega}=\tfrac{\omega}{A}$. To apply concavity along the line segment between $(0,0)$ and $(\widetilde{\omega}, 1)$, we find $\theta\in [0,1]$ such that
    $$
    (\omega, A)=(1-\theta)\cdot (0,0)+\theta \cdot (\widetilde{\omega}, 1),
    $$
so $\theta=A$. By concavity and non-negativity of $M$, and by \eqref{e:newInfo-0}: 
    $$
    M(\omega, A) \geq (1-\theta)\cdot M(0,0)+\theta \cdot M(\widetilde{\omega}, 1) \geq A.
    $$
So we have that:
    $$
    \text{For all } M\in\scm: \: M(\omega, A) \geq A \text{ for all } (\omega, A) \in \Sigma_1.
    $$
We therefore set our candidate $\bfm(\omega,A):=A$ throughout the $\Sigma_1$ region. 

Similarly, the ray through the origin and a point 
$(\omega, A) \in \Delta_0$, intersects the line segment $[P_1; P_2]$ in the point $(\widetilde{\omega}, \widetilde{A})$, and $(\omega, A)=(1-\theta)\cdot (0,0)+\theta \cdot (\widetilde{\omega}, \widetilde{A})$ then holds for $\theta=\omega \widetilde{\omega}^{-1}$. Concavity again yields $M(\omega, A) \geq \theta$, and a straightforward calculation shows that:
    $$
    \text{For all } M\in\scm: \: M(\omega, A) \geq \frac{\omega+(1-\omega_0)A}{2-\omega_0} =: \bfm(\omega, A) \: \text{ for all } (\omega, A) \in \Delta_0.
    $$

Finally, consider a point $(\omega, A) \in \Delta$. An analogous computation -- this time interpolating along line segments between $(0,0)$ and $(\widetilde{\omega}, 2)$ -- yields:
    \begin{equation}
    \label{e:Buildf-C}
    \text{For all } M\in\scm: \: M(\omega, A) \geq \frac{A}{2} \: M\Big(\frac{2\omega}{A}, 2\Big) \: \text{ for all } (\omega, A) \in \Delta.
    \end{equation}
So all we have left to do is define
    $$
    f(\omega):= M(\omega, 2), \:\text{ for }\: 0<\omega< \omega_0.
    $$

Our starting point is $f(\omega_0)=1$.
We will propagate this data in the following way: first, \eqref{e:Buildf-C} tells us that
    $$
    M\left(\frac{\omega}{2}, 1\right)\geq \frac{1}{2} f(\omega), \text{ for all } \omega\in(0,\omega_0),
    $$
i.e. information at level $A=2$ generates data at level $A=1$ via concavity. In turn, information at level $A=1$ can be ``jumped,'' to produce data at level $A=2$:
    $$
    f\left(\varphi\left(\frac{\omega}{2}\right)\right) \geq M\left(\frac{\omega}{2}, 1\right)\geq \frac{1}{2} f(\omega), \text{ for all } \omega\in(0,\omega_0).
    $$
We will see the sequence in \eqref{e:def-omegar-exp} emerging as a result.

\vspace{0.1in}

\begin{minipage}[c]{0.5\linewidth}
Run the initial data at $\omega_0$ through this process:
    $$
    f\left(\varphi\left(\frac{\omega_0}{2}\right)\right) \geq \frac{1}{2} f(\omega_0)=\frac{1}{2},
    $$
or
    $$
    f(\omega_1)\geq \frac{1}{2}, \text{ where } \omega_1:= \varphi\left(\frac{\omega_0}{2}\right).
    $$

Inductively, we obtain a decreasing sequence $\{\omega_n\}_{n\geq 0}$, defined recursively by
$$
      \omega_n:= \varphi\left(\frac{\omega_{n-1}}{2}\right) \text{ for } n\geq 1,
$$
such that $f(\omega_n) \geq 2^{-n}$, for all $n\geq 0$.
\end{minipage}
\begin{minipage}[c]{0.5\linewidth}
\begin{center}
\includegraphics[width=\linewidth]{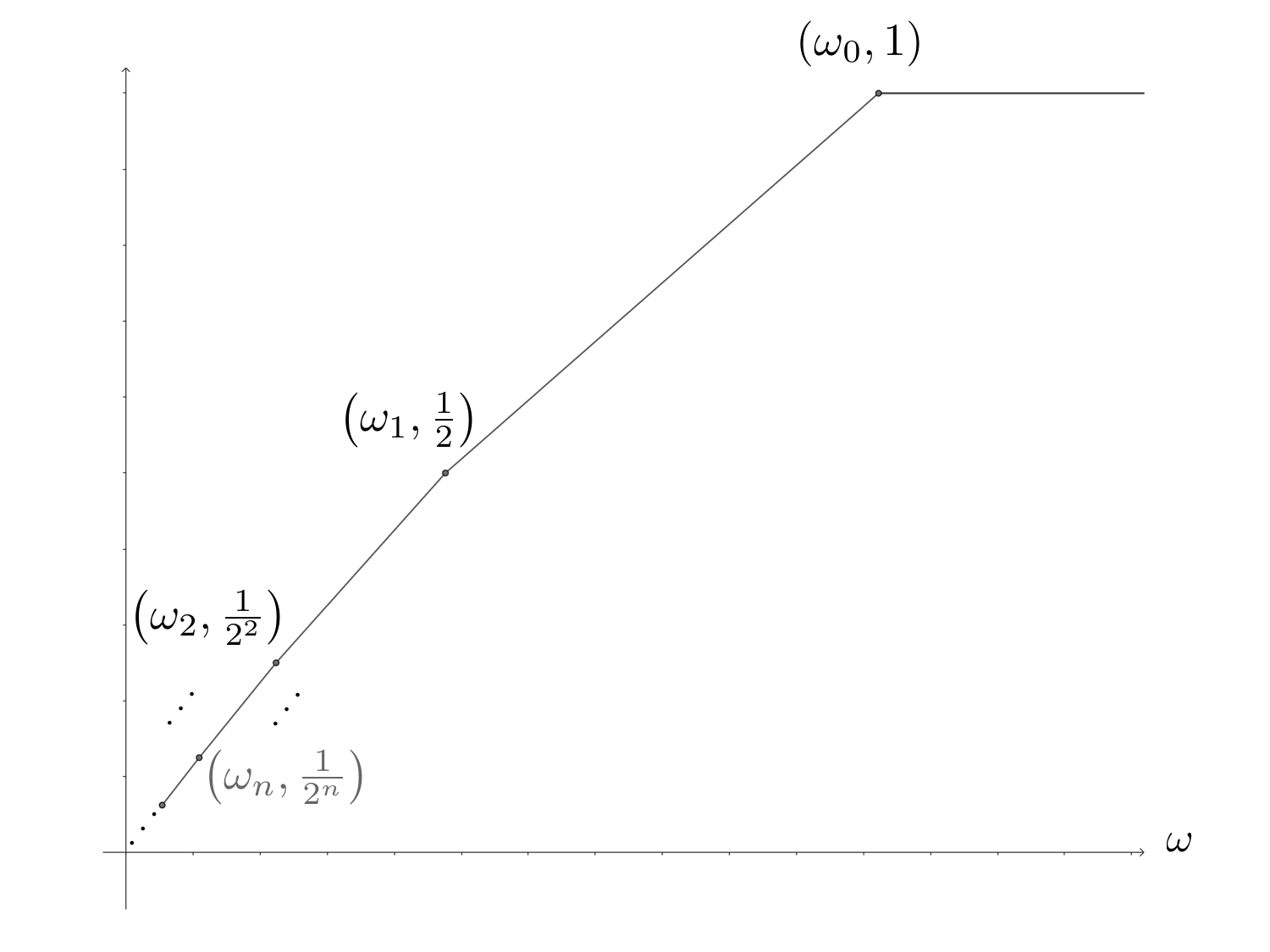}
\end{center}
\end{minipage}


\noindent Then, since $f$ is concave, it must lie above the line segments determined by the points $(\omega_n, 2^{-1/n})$. We denote this by
    $$
    f(\omega) \geq \Lint\bigg[(\omega_0, 1); \: (\omega_1, 2^{-1}); \: \ldots ;\: (\omega_n, 2^{-n}); \ldots\bigg],
    $$
for all $\omega\in(0,\omega_0)$. We therefore arrive at the function defined in \eqref{e:f-def-r}:
    \begin{equation}
        \label{e:f-def}
    \mathfrak{f}: [0,\infty) \rightarrow [0,1]; \:\:\:
    \mathfrak{f}(\omega) := \begin{cases}
        0, & \:\text{ if }\: \omega=0\\
        \Lint[(\omega_n, 2^{-n}); \: n\in \mbz_{\geq 0}], & \:\text{ if }\: 0<\omega<\omega_0\\
        1, & \:\text{ if }\: \omega \geq \omega_0.
    \end{cases}
    \end{equation}
Explicitly, for $0<\omega<\omega_0$:
    \begin{equation}
    \label{e:f-explicit}
        \text{If }\:\: \omega_{n}\leq \omega< \omega_{n-1}, \:\text{ for } n\geq 1, \text{ then: }\:\: \mathfrak{f}(\omega):= \frac{1}{2^n} \left(1+\frac{\omega-\omega_{n}}{\omega_{n-1} -\omega_{n}}\right).
    \end{equation}

Define $\bfm: \Omega_\mbm \rightarrow [0,1]$ by  
\begin{equation}
\label{e:M-def}
	\bfm(\omega,A):=\begin{cases}
	\frac{A}{2}\mathfrak{f}\left(\frac{2\omega}{A}\right), & \text{ if } (\omega, A) \in \Delta;\\[7pt]
        \frac{\omega+(1-\omega_0)A}{2-\omega_0}, & \text{ if } (\omega, A) \in \Delta_0; \\[7pt]
        A, & \text{ if } (\omega, A) \in \Sigma_1; \\[7pt]
        1, & \text{ if } (\omega, A) \in \Sigma_0.
	\end{cases}
\end{equation}

\noindent We have proved:

\begin{theorem}
\label{T:bff-minimizes}
    The function $\bfm$ defined in \eqref{e:M-def} is a minimizer of the collection $\scm$, that is
        $$
        \bfm(\omega, A) \leq M(\omega, A),
        $$
    for all $(\omega,A)\in\Omega_\mbm$ and all $M\in\scm$.
\end{theorem}

\section{Proving the candidate $\bfb$ is the true Bellman function.}
\label{S:ProofTrue}


The candidate $\bfm$ satisfies the Obstacle \eqref{e:OC-M} by construction, and since it is the minimum of countably many linear functions, $\bfm$ is concave:

\begin{lemma} 
    The infimum of countably many concave functions $\psi_n:\mbr^d\to\mathbb{R}$ (assumed to be defined on the whole domain) is concave, provided the infimum is finite.
\end{lemma}


\begin{proof}
    
 Define $\psi:\mbr^d\to\mbr$ by $\psi(y)=\inf\limits_{n\in\mbn}\psi_n(y)$, and let $y_1,y_2\in\mbr^d$. Then, consider the line segment between the points $t\psi(y_1) + (1-t)\psi(y_2),t\in[0,1]$. Observe that by restricting to the convex hull of the two points, the problem is reduced to proving that the infimum of finitely many concave univariate functions $\mbr\to\mbr$ is concave.

Now consider $\psi(ty_1+(1-t)y_2)=\inf\limits_{n}\psi_n(ty_1+(1-t)y_2)$. Because $\psi$ is an infimum, it is sufficient to prove that $t\psi(y_1)+(1-t)\psi(y_2)$ is a lower bound of the quantity $\psi_n(ty_1+(1-t)y_2)$ for each $n$. By definition,

$$
t\psi(y_1) + (1-t)\psi(y_2)=t\inf\limits_{m}\psi_m(y_1) + \inf\limits_{n}\psi_n(y_2)
$$

By properties of an infimum, i.e. because $\psi(y_i)\leq\psi_j(y_i)$ for all $j$. Hence for each $n$,

$$
t\psi(y_1) + (1-t)\psi(y_2)\leq t\psi_n(y_1) + (1-t)\psi_n(y_2)
$$

Then applying concavity of the $\psi_n$'s,

$$
t\psi(y_1) + (1-t)\psi(y_2)\leq \psi_n(ty_1 + (1-t)y_2)
$$

Now using properties of the infimum as described in the second paragraph of the proof, we are done. 



\end{proof}

\begin{center}
  {\Large\ding{167}} 
\end{center}


We are only left to show that $\bfm$ satisfies the Jump Inequality \eqref{e:M-J}. 
We will rely on properties of the $\varphi(\omega) := \varphi_r(\omega)$ function defined in \eqref{e:varphi-def}.
It is easy to see that 
$\varphi$ is strictly increasing, strictly concave, and satisfies 
    $$ \varphi(0)=0; \:\:\:\: \lim_{\omega\rightarrow\infty} \varphi(\omega) = 1;
    \:\:\:\: \varphi(\omega) \leq \omega.$$

\begin{proposition}
        \label{P:Jump1}
The candidate defined in \eqref{e:M-def} satisfies:
    \begin{equation}
        \label{e:Jump-bfm}
    \bfm\big(\varphi(\omega), A+1\big) \geq \bfm(\omega, A),
    \end{equation}
for all $\omega \geq 0$ and $A\in [0,1]$.
\end{proposition}

\begin{proof}
Note that \eqref{e:Jump-bfm} holds trivially for $\omega=0$ and $A=0$, so assume below that $\omega >0$ and $A\in (0,1]$.
First, we prove \eqref{e:Jump-bfm} holds at the endpoint $A=1$, namely 
    \begin{equation}
        \label{e:Jump-bfm-1}
        \mathfrak{f}\bigg( \varphi(\omega)\bigg) \geq \bfg(\omega) := \bfm(\omega, 1),
        \:\:\: \text{ for all } \omega \geq 0.
    \end{equation}
Recall that $\mathfrak{f}$ is defined in \eqref{e:f-def}. We include here  for completeness the expression for $\bfg$:
 \begin{equation}
        \label{e:g-def}
    \bfg: [0,\infty) \rightarrow [0,1]; \:\:\:
    \bfg(\omega) := \begin{cases}
        \frac{1}{2}\mathfrak{f}(2\omega), & \:\text{ if }\: 0 \leq \omega < \frac{\omega_0}{2}\\[7pt]
        \frac{\omega+1-\omega_0}{2-\omega_0}, & \:\text{ if }\: \frac{\omega_0}{2} \leq \omega<1\\[7pt]
        1, & \:\text{ if }\: \omega \geq 1.
    \end{cases}
    \end{equation}
Note that, for $\omega <1$, the function $\bfg$ is just linear interpolation between the points 
    $$
    \left(1,1\right); \:\:
    \left(\frac{\omega_0}{2},\frac{1}{2}\right); \:\:
     \left(\frac{\omega_1}{2},\frac{1}{2^2}\right); \:\:
     \ldots
      \left(\frac{\omega_{n-1}}{2},\frac{1}{2^{n}}\right); \:\:
      \ldots
    $$
Moreover, $\varphi$ maps the $\omega$-coordinates of these points exactly to the $\omega$-coordinates of the critical points on the graph of $\mathfrak{f}$:
    $$
    \left(\omega_0, 1\right), \:\:
    \left(\omega_1, \frac{1}{2}\right), \:\:
    \left(\omega_2, \frac{1}{2^2} \right), \:\:
    \ldots
        \left(\omega_n, \frac{1}{2^n} \right), \:\:
    \ldots.
    $$

To prove \eqref{e:Jump-bfm-1}, consider the function:
       $$
        h(\omega) := \mathfrak{f}\big( \: \varphi(\omega) \: \big) - \bfg(\omega); \:\:\:\: \omega>0.
        $$

\begin{itemize}
    \item If $\omega\geq 1$, then $\varphi(\omega) \geq \varphi(1)=\omega_0$, so both sides of  \eqref{e:Jump-bfm-1} are equal to $1$.

    \item For $\frac{\omega_0}{2} \leq \omega \leq 1$, we have $\omega_1 \leq \varphi(\omega) \leq \omega_0$, and from \eqref{e:f-explicit} and \eqref{e:g-def}, $h$ takes the explicit form
        $$
        h(\omega) = \frac{1}{2} \left(1 + \frac{\varphi(\omega) - \omega_1}{\omega_0-\omega_1}\right) - \frac{\omega+1-\omega_0}{2-\omega_0}.
        $$
    Then for all $\omega \in (\omega_0/2, \: 1)$:
        $$
        h''(\omega) = \frac{1}{2(\omega_0-\omega_1)} \varphi''(\omega) <0,
        $$
    so $h$ is concave on this interval. Moreover, $h$ is $0$ at the endpoints:
        $$
        h\left(\frac{\omega_0}{2}\right) = \mathfrak{f}(\omega_1) - \bfg\left(\frac{\omega_0}{2}\right) = \frac{1}{2} - \frac{1}{2} = 0; 
        \:\:\:\:\:\:\:
        h(1) = \mathfrak{f}(\omega_0)-\bfg(1) = 1-1=0,
        $$
    so $h(\omega) \geq 0$ for all $\omega \in [\omega_0/2, \: 1]$.

    \item Finally, let $\omega \in (0, \tfrac{\omega_0}{2})$. Let $n\geq 1$ be such that
        $$
        \frac{\omega_n}{2} \leq \omega < \frac{\omega_{n-1}}{2}.
        $$
    Then
        $$
        h(\omega) = \frac{1}{2^{n+1}}\left(1 + \frac{\varphi(\omega)-\omega_{n+1}}{\omega_n - \omega_{n+1}}\right) -
        \frac{1}{2^{n+1}}\left(1 + \frac{2\omega-\omega_{n}}{\omega_{n-1} - \omega_{n}}\right).
        $$
    Similarly, $h''(\omega) <0$ for all $\omega\in (\omega_n/2, \: \omega_{n-1}/2)$, so $h$ is concave on this interval. Once again, it is $0$ at the endpoints: 
        $$
        h\left(\frac{\omega_n}{2}\right) = h\left(\frac{\omega_n}{2}\right) = 0,
        $$
    so $h\geq 0$ on each interval of the form $[\omega_n/2, \: \omega_{n-1}/2)$, and thus throughout $(0, \omega_0/2)$.
\end{itemize}

Lastly, we show that \eqref{e:Jump-bfm} holds for any $\omega>0$ and  $A\in (0,1)$:
    $\bfm(\: J(\omega, A)\:)\geq \bfm(\omega, A),$
where $J$ denotes the ``jump transformation''
    $$
    J(\omega, A) := (\varphi(\omega), \: A+1).
    $$
As shown in Figure \ref{fig:Jregs}, we split the lower half of the domain into three regions, depending on where in the upper half $J(\omega, A)$ lands.

\begin{figure}[h]
\begin{center}
\includegraphics[scale=0.42]{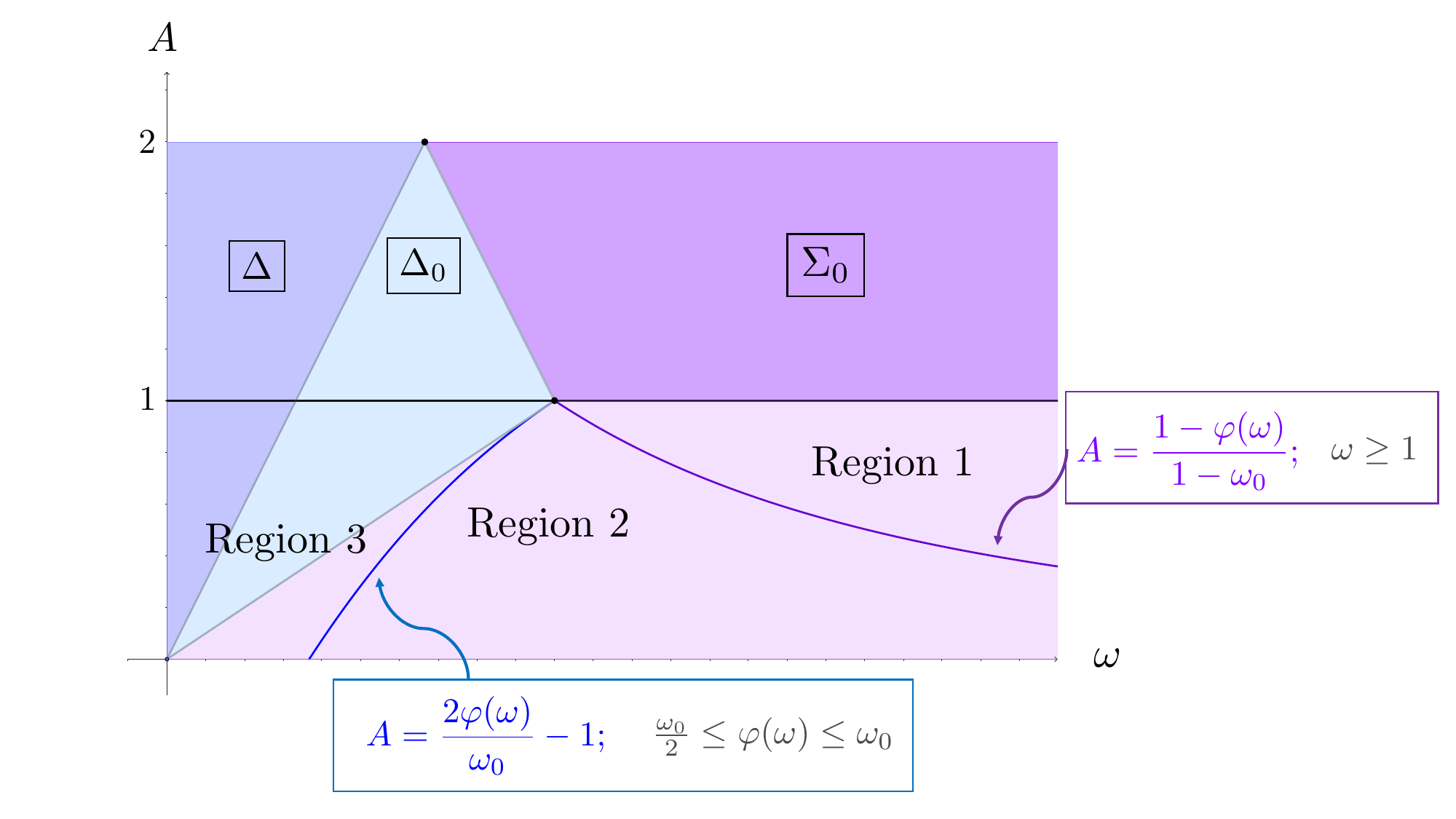}
\end{center}
\caption{Jump regions}
\label{fig:Jregs}
\end{figure}

\vspace{0.2in}
\textbf{Region 1:} 
    $$
    (\omega, A): \:\: \omega>0; \:\: \frac{1-\varphi(\omega)}{1-\omega_0} \leq A <1.
    $$
Then $J(\omega, A)$ lands in $\Sigma_0$ and \eqref{e:Jump-bfm} becomes $1 \geq A$, which is certainly true.

\vspace{0.2in}

\textbf{Region 2:} 
    $$
    (\omega, A): \:\: \omega>0; \:\:
    0< A \leq \min\left(\frac{2\varphi(\omega)}{\omega_0}-1, \: \frac{1-\varphi(\omega)}{1-\omega_0}\right).
    $$
Then $J(\omega, A)$ lands in $\Delta_0$, so
    $$
    \bfm(\: J(\omega, A)\:) = \frac{\varphi(\omega) +(1-\omega_0)(A+1)}{2-\omega_0},
    $$
and \eqref{e:Jump-bfm} becomes equivalent to
    \begin{equation}
        \label{e:temp1}
    A \leq \varphi(\omega) + 1 - \omega_0.
    \end{equation}
But remember that $A$ is bounded by the minimum between two curves.
If $\varphi(\omega) \leq \omega_0$, then
    $$
    \varphi(\omega) + 1 - \omega_0 \geq \frac{2\varphi(\omega)}{\omega_0}-1,
    $$    
otherwise 
    $$
    \varphi(\omega) + 1 - \omega_0 \geq \frac{1-\varphi(\omega)}{1-\omega_0}.
    $$
Either way, \eqref{e:temp1} holds.

\vspace{0.2in}
\textbf{Region 3:} 
    $$
    (\omega, A): \:\: 0 < \omega < \varphi^{-1}\left(\frac{A+1}{2}\omega_0\right); \:\: 0<A<1.
    $$
Then $J(\omega, A)$ lands in $\Delta$, so
    $$
    \bfm(\: J(\omega, A)\:) = \frac{A+1}{2} \cdot \mathfrak{f}\left( \frac{2\varphi(\omega)}{A+1} \right).
    $$
Observe now that we can also express $\bfm(\omega, A)$ in terms of $\bfg$:
    $$
    \bfm(\omega, A) = A \cdot \bfg\left(\frac{\omega}{A}\right).
    $$
Therefore \eqref{e:Jump-bfm} becomes equivalent to
     \begin{equation}
         \label{e:temp2}
         \mathfrak{f}\left( \frac{2\varphi(\omega)}{A+1} \right) \geq \frac{2A}{A+1} \cdot \bfg\left(\frac{\omega}{A}\right).
     \end{equation}

Recall that for a concave function $h:[0,\infty)\rightarrow\mbr$ with $h(0)\geq 0$, we have
    $$
    h(\theta x) \geq \theta \cdot h(x), \text{ for all } \theta \in [0,1].
    $$
Since $\varphi \geq 0$ is concave and $A\in(0,1)$, 
    $$
    \varphi(\omega) = \varphi\left(A \cdot \frac{\omega}{A}\right) \geq A \cdot \varphi\left(\frac{\omega}{A}\right).
    $$
Then, because $\mathfrak{f}$ is non-decreasing,
    \begin{align*}
        \mathfrak{f}\left( \frac{2\varphi(\omega)}{A+1} \right) &\geq 
    \mathfrak{f}\left( \frac{2A}{A+1} \cdot \varphi\left(\frac{\omega}{A}\right)\right)\\
        & \geq \frac{2A}{A+1} \cdot \mathfrak{f}\left(
        \varphi\left(\frac{\omega}{A}\right)\right)\\
        & \geq \frac{2A}{A+1} \cdot \bfg\left(\frac{\omega}{A}\right),
    \end{align*}
where the second inequality follows by concavity of $\mathfrak{f}$ and $\tfrac{2A}{A+1}<1$, and the third inequality follows from \eqref{e:Jump-bfm-1}. This proves \eqref{e:temp2}, and we are done.
\end{proof}

\begin{center}
  {\Large\ding{167}} 
\end{center}
\newpage



\bibliography{bibliography}

@book {MR3363413,
    AUTHOR = {Donoghue, Jr., William F.},
     TITLE = {Distributions and {F}ourier transforms},
    SERIES = {Pure and Applied Mathematics},
    VOLUME = {32},
 PUBLISHER = {Academic Press, New York},
      YEAR = {1969},
     PAGES = {viii+315},
   MRCLASS = {42-01 (28-01 30-01 60-01)},
  MRNUMBER = {3363413},
}

@article {HRS1,
    AUTHOR = {Holmes Fay, Irina and Rey, Guillermo and \v{S}kreb, Kristina
              Ana},
     TITLE = {Sharp restricted weak-type estimates for sparse operators},
   JOURNAL = {J. Funct. Anal.},
  FJOURNAL = {Journal of Functional Analysis},
    VOLUME = {288},
      YEAR = {2025},
    NUMBER = {11},
     PAGES = {Paper No. 110854, 45},
      ISSN = {0022-1236,1096-0783},
   MRCLASS = {42B20 (42B25)},
  MRNUMBER = {4864875},
       DOI = {10.1016/j.jfa.2025.110854},
       URL = {https://doi.org/10.1016/j.jfa.2025.110854},
}

@article {LNBook,
    AUTHOR = {Lerner, Andrei K. and Nazarov, Fedor},
     TITLE = {Intuitive dyadic calculus: the basics},
   JOURNAL = {Expo. Math.},
  FJOURNAL = {Expositiones Mathematicae},
    VOLUME = {37},
      YEAR = {2019},
    NUMBER = {3},
     PAGES = {225--265},
      ISSN = {0723-0869,1878-0792},
   MRCLASS = {42B20 (42B25)},
  MRNUMBER = {4007575},
MRREVIEWER = {Luz\ Roncal},
       DOI = {10.1016/j.exmath.2018.01.001},
       URL = {https://doi.org/10.1016/j.exmath.2018.01.001},
}

@incollection{Reznikov2019,
AUTHOR={Reznikov, Alexander
and Volberg, Alexander},
EDITOR={Aleman, Alexandru
and Hedenmalm, Haakan
and Khavinson, Dmitry
and Putinar, Mihai},
TITLE={The Weak Type Estimates of Two Different Martingale Transforms Coincide},
BOOKTITLE={Analysis of Operators on Function Spaces: The Serguei Shimorin Memorial Volume},
YEAR={2019},
PUBLISHER={Springer International Publishing},
ADDRESS={Cham},
PAGES={259--274},
ISBN={978-3-030-14640-5},
DOI={10.1007/978-3-030-14640-5_11},
URL={https://doi.org/10.1007/978-3-030-14640-5_11}
}

@article {MR3812865,
    AUTHOR = {Nazarov, Fedor and Reznikov, Alexander and Vasyunin, Vasily
              and Volberg, Alexander},
     TITLE = {On weak weighted estimates of the martingale transform and a
              dyadic shift},
   JOURNAL = {Anal. PDE},
  FJOURNAL = {Analysis \& PDE},
    VOLUME = {11},
      YEAR = {2018},
    NUMBER = {8},
     PAGES = {2089--2109},
      ISSN = {2157-5045,1948-206X},
   MRCLASS = {42A45 (42A50 42A61 42B20 42B35 42B37 47A30 49L25)},
  MRNUMBER = {3812865},
MRREVIEWER = {Luz\ Roncal},
       DOI = {10.2140/apde.2018.11.2089},
       URL = {https://doi.org/10.2140/apde.2018.11.2089},
}

@incollection {MR2500520,
    AUTHOR = {Vasyunin, Vasily and Volberg, Alexander},
     TITLE = {Monge-{A}mp\`ere equation and {B}ellman optimization of
              {C}arleson embedding theorems},
 BOOKTITLE = {Linear and complex analysis},
    SERIES = {Amer. Math. Soc. Transl. Ser. 2},
    VOLUME = {226},
     PAGES = {195--238},
 PUBLISHER = {Amer. Math. Soc., Providence, RI},
      YEAR = {2009},
      ISBN = {978-0-8218-4801-2; 0-8218-4801-1},
   MRCLASS = {42B37 (35J60 41A44 60D05)},
  MRNUMBER = {2500520},
MRREVIEWER = {Wilfrid\ Gangbo},
       DOI = {10.1090/trans2/226/16},
       URL = {https://doi.org/10.1090/trans2/226/16},
}
\bibliographystyle{abbrv}

\section*{Acknowledgments}
This work was completed during the Spring 2025 semester at the University of Wyoming, as part of the Topics Course Math 5290. I.H.F. is supported by NSF grant NSF-DMS-2246985.

\vspace{0.2in}

\small
\begin{tabular}{ll}
\textbf{Irina Holmes Fay} & University of Wyoming,
\:\: \texttt{iholmesf@uwyo.edu} \\
\textbf{Zachary H. Pence} & University of Wyoming, 
\:\: \texttt{zpence@uwyo.edu} \\
\textbf{John Freeland Small} & University of Wyoming, \:\: \texttt{jsmall6@uwyo.edu} \\
\textbf{Xiaokun Zhou} & University of Wyoming, 
\:\: \texttt{xzhou3@uwyo.edu}
\end{tabular}

\end{document}